\documentclass[a4paper,10pt]{amsart}
\usepackage[a4paper,tmargin=40mm,bmargin=35mm,hmargin=30mm]{geometry}
\usepackage[utf8]{inputenc}
\usepackage[T1]{fontenc}

\usepackage{amsmath}
\usepackage{amssymb}
\usepackage{amsthm}
\usepackage{mathtools}

\usepackage{eucal} 
\usepackage{mathdots}
\usepackage[colorlinks=true,linkcolor=red,citecolor=blue,urlcolor=green]{hyperref}
\usepackage[capitalise]{cleveref}
\usepackage{enumitem}
\usepackage{mathrsfs}  

\usepackage{tikz-cd}
\usepackage{adjustbox}
\usepackage{contour}
\usepackage{ulem}

\contourlength{0.8pt}

\newcommand{\ULL}[1]{%
  \uline{\phantom{#1}}%
  \llap{\contour{white}{#1}}%
}

\newtheorem{theorem}{Theorem}[section]
\newtheorem{prop}[theorem]{Proposition}
\newtheorem{cor}[theorem]{Corollary}
\newtheorem{lemma}[theorem]{Lemma}

\theoremstyle{definition}
\newtheorem{dfn}[theorem]{Definition}
\newtheorem{rmk}[theorem]{Remark}
\newtheorem{ex}[theorem]{Example}

\DeclareMathOperator{\Hom}{\mathsf{Hom}}
\DeclareMathOperator{\End}{\mathsf{End}}
\DeclareMathOperator{\dgEnd}{\mathsf{dgEnd}}

\DeclareMathOperator{\RHom}{\mathbf{R}\mathsf{Hom}}

\newcommand{\Acal}{\mathcal{A}}

\newcommand{\Ccal}{\mathcal{C}}
\newcommand{\Dcal}{\mathcal{D}}
\newcommand{\Ecal}{\mathcal{E}}
\newcommand{\Fcal}{\mathcal{F}}
\newcommand{\Gcal}{\mathcal{G}}
\newcommand{\Hcal}{\mathcal{H}}

\newcommand{\Rcal}{\mathcal{R}}

\newcommand{\Tcal}{\mathcal{T}}
\newcommand{\Ucal}{\mathcal{U}}
\newcommand{\Vcal}{\mathcal{V}}

\newcommand{\Qbb}{\mathbb{Q}}

\newcommand{\Tbb}{\mathbb{T}}
\newcommand{\Zbb}{\mathbb{Z}}

\newcommand{\Sfr}{\mathfrak{S}}
\newcommand{\Ifr}{\mathfrak{I}}
\newcommand{\Rfr}{\mathfrak{R}}
\newcommand{\Mfr}{\mathfrak{M}}
\newcommand{\Pfr}{\mathfrak{P}}

\newcommand{\D}{\mathsf{D}}

\newcommand{\cpt}{\mathsf{c}}
\newcommand{\bdd}{\mathsf{b}}
\newcommand{\op}{\mathsf{op}}

\newcommand{\K}{\mathsf{K}}

\newcommand{\Mod}[1]{\mathsf{Mod}\mbox{-}#1}
\newcommand{\lMod}[1]{#1\mbox{-}\mathsf{Mod}}
\newcommand{\dgMod}[1]{\mathsf{dgMod}\mbox{-}#1}
\newcommand{\dglMod}[1]{#1\mbox{-}\mathsf{dgMod}}
\newcommand{\Proj}{\mathsf{proj}}
\newcommand{\Modpr}[1]{\mathsf{Mod}_{\Proj}\mbox{-}#1}
\newcommand{\Ctra}[1]{\mathsf{Ctra}\mbox{-}#1}
\newcommand{\Ctrapr}[1]{\mathsf{Ctra}_{\Proj}\mbox{-}#1}
\newcommand{\Disc}[1]{#1\mbox{-}\mathsf{Discr}}

\renewcommand{\mod}[1]{\mathsf{mod}\mbox{-}#1}
\newcommand{\real}{\mathsf{real}}

\newcommand{\Spec}[1]{\mathsf{Spec}(#1)}

\newcommand*{\Perp}[1]{{}^{\perp_{#1}}}

\newcommand{\Sets}{\mathsf{Sets}}
\newcommand{\Prod}{\mathsf{Prod}}
\newcommand{\Def}{\mathsf{Def}}
\newcommand{\thick}{\mathsf{thick}}
\newcommand{\Cone}{\mathsf{Cone}}

\newcommand{\Coker}{\mathsf{Coker}}
\newcommand{\Ann}{\mathsf{Ann}}
\newcommand{\Add}{\mathsf{Add}}

\newcommand{\trace}{\mathsf{tr}}

\newcommand{\pp}{\mathfrak{p}}

\newcommand{\mm}{\mathfrak{m}}

\newcommand{\height}{\mathsf{ht}}
\newcommand{\yo}{\mathbf{y}}
\newcommand{\toeq}{\xrightarrow{\cong}}
\newcommand{\ctrtensor}{\odot}
\newcommand{\RGamma}{\mathbf{R}\Gamma}
\newcommand{{\tst}}{\textit{t}-}

\newcommand{\newterm}[1]{\ULL{\text{#1}}}

\title{Topological endomorphism rings of tilting complexes}
\author{Michal Hrbek}
\address[M. Hrbek]{Institute of Mathematics of the Czech Academy of Sciences, \v{Z}itn\'{a} 25, 115 67 Prague, Czech Republic}
\email{hrbek@math.cas.cz}


\thanks{The author was supported by the GAČR project 20-13778S and RVO: 67985840.}
\begin{document}
\begin{abstract}
    In a compactly generated triangulated category, we introduce a class of tilting objects satisfying certain purity condition. We call these the decent tilting objects and show that the tilting heart induced by any such object is equivalent to a category of contramodules over the endomorphism ring of the tilting object endowed with a natural linear topology. This extends the recent result for n-tilting modules by Positselski and Šťovíček. In the setting of the derived category of modules over a ring, we show that the decent tilting complexes are precisely the silting complexes such that their character dual is cotilting. The hearts of cotilting complexes of cofinite type turn out to be equivalent to the category of discrete modules with respect to the same topological ring. Finally, we provide a kind of Morita theory in this setting: Decent tilting complexes correspond to pairs consisting of a tilting and a cotilting derived equivalence as described above tied together by a tensor compatibility condition.
\end{abstract}
\maketitle
\tableofcontents
\section{Introduction}
In the landmark result \cite{Ri89}, Rickard established a full Morita theory for derived categories of modules: There is a triangle equivalence $\D^\bdd(\Mod R) \cong \D^\bdd(\Mod S)$ between the bounded derived categories of right modules if and only if there is a compact tilting complex $T$ in $\D^\bdd(\Mod R)$ with $S \cong \End_{\D(\Mod R)}(T)$. In addition, the triangle equivalence can be represented by the derived functor $\RHom_R(T,-)$; this fact was later explained in terms of the endomorphism dg-ring of $T$ by Keller \cite{Kel93}. The whole picture can be made left-right symmetric following the observation made in \cite{Ri91}: $T$ also induces a triangle equivalence $T \otimes_R^\mathbf{L} -: \D^\bdd(\lMod R) \cong \D^\bdd(\lMod S)$ on the side of left modules, and the two equivalences are compatible with the tensor products in the sense that the following diagram commutes:
\begin{equation}\label{I1}
\begin{tikzcd}
    \D^\bdd(\Mod R) \times \D^\bdd(\lMod R) \arrow{d}[swap]{(\RHom_R(T,-), T \otimes_R^\mathbf{L} -)}{\cong} \arrow{r}{- \otimes_R^\mathbf{L} -} & \D(\Mod \Zbb) \arrow{d}{=} \\
    \D^\bdd(\Mod S) \times \D^\bdd(\lMod S) \arrow{r}{- \otimes_{S}^\mathbf{L} -} & \D(\Mod \Zbb)
\end{tikzcd}
\end{equation}

More recently, efforts were made to see to which extent the theory can be stated for tilting objects which are not necessarily compact. For a survey of both the past and the very recent development of the theory of these ``large'' silting and tilting objects we refer to \cite{AH19}. When the tilting complex $T$ is not compact, it cannot be expected to induce a derived equivalence between its endomorphism ring and $R$. However, Bazzoni showed in \cite{Bazz10} that if $T$ is a 1-tilting module which has the additional property of being ``good'', then the derived category of $R$ embeds via $\RHom_R(T,-)$ to the derived category of $\End_R(T)$, and in fact, this fully faithful functor is a part of a recollement of triangulated categories. The assumption of being good is very mild in the sense that every tilting module is additively equivalent to a good one. This was later extended to general ($n$-)tilting modules by Bazzoni, Mantese, and Tonolo \cite{BMT11}, and even further to a general setting of dg categories by Nicolás and Saorín \cite{NS18}.

A very recent breakthrough came in the paper \cite{PS21}. There, the endomorphism ring of a tilting module $T$ is endowed with a linear topology such that the resulting topological ring $\Sfr = \End_{R}(T)$ is complete and separated. Such a structure comes associated with an abelian category of right $\Sfr$-contramodules. This category can be morally viewed as the well-behaved replacement of the ill-behaved category of complete and separated $\Sfr$-modules. The theory of contramodules over complete and separated topological rings, developed in the last decade chiefly by Positselski, has quickly found many strong applications in algebra and algebraic geometry, see e.g. \cite{Pos12} and the survey \cite{contramodules}. Positselski and Šťovíček showed in \cite{PS21} that the heart of the tilting {{\tst}}structure induced by $T$ is equivalent to the category $\Ctra \Sfr$ of right $\Sfr$-contramodules. Assuming again the mild additional condition of $T$ being good, the forgetful functor $\Ctra \Sfr \to \Mod \Sfr$ is fully faithful, both on the abelian and the derived level. In particular, the essential image of the fully faithful functor $\RHom_R(T,-)$ in the setting of \cite{BMT11} is now given an algebraic description --- it is the derived category of contramodules over the endomorphism ring.

The aim of the present paper can be described as an attempt to find a common generalization of the results of Rickard and of Positselski-Šťovíček. The first goal is to extend the latter result from tilting modules to tilting complexes, or more generally, to silting objects in the sense of Psaroudakis and Vitória \cite{PV18} and Nicolás, Saorín, and Zvonareva \cite{NSZ19}. It turns out that this cannot be achieved without additional assumptions --- there are tilting complexes whose tilting heart cannot be equivalent to a category of contramodules over any complete and separated topological ring (\cref{counterexample}). Therefore, we are forced to find a condition on a tilting complex which guarantees decent behavior. The condition we consider comes from the purity theory of compactly generated triangulated categories as established by Krause \cite{Kr00}, and rely on the new techniques using Grothendieck derivators developed recently by Laking \cite{Lak20}. In Section 2, we introduce the notion of a decent tilting object as a silting object such that its definable closure is contained in the silting heart. Such objects are automatically tilting in the sense of \cite{PV18}, and therefore are expected to induce derived equivalences. In the generality of a compactly generated triangulated category $\Tcal$, we show that the heart of the {{\tst}}structure induced by a decent tilting object is equivalent to a contramodule category of the endomorphism ring $\Sfr$ endowed with a suitable triangulated replacement of the finite topology we call the compact topology (\cref{heartequiv}, \cref{ab-eq}). The way we prove this is by using the restricted Yoneda functor to reduce the problem from the triangulated setting to the abelian setting of modules over a ringoid, where the results of \cite{PS21} apply directly. Under mild assumptions on $\Tcal$, the existence of a decent tilting complex in $\Tcal$ then yields a triangle equivalence $\Tcal^\bdd \cong \D^\bdd(\Ctra \Sfr)$, where $\Tcal^\bdd$ is the subcategory of suitably bounded objects of $\Tcal$ (\cref{gen-der-eq}). 

In Section 3, we specialize to the setting of the derived category of modules over a ring. There, it turns out that our condition has a very natural interpretation: A silting complex $T$ of right $R$-modules is decent if and only if its character dual is a cotilting complex of left $R$-modules (\cref{condition}). It follows that our notion of a decent tilting complex includes all tilting modules and all compact tilting complexes, and the decent tilting complexes correspond bijectively to cotilting complexes of cofinite type via the character duality. Moreover, we discuss some recently studied sources of interesting examples coming from commutative algebra \cite{PV20}, \cite{HNS}.

Motivated by the aforementioned results of \cite{BMT11}, we show in Section 4 that, up to additive equivalence, $T$ can be assumed to be ``good'' --- a technical condition which allows to represent the derived equivalence $\D^\bdd(\Mod R) \toeq \D^\bdd(\Ctra \Sfr)$  by the derived functor $\RHom_R(T,-)$. In Section 5, we suitable dualize this to show that the cotilting heart associated to the character dual of a decent tilting complex is equivalent to another category attached to the topological ring --- the category $\Disc \Sfr$ of left discrete $\Sfr$-modules (\cref{cotilting-heart}). This gives an explicit description of the hearts induced by cotilting complexes of cofinite type (\cref{cotilting-cofinite}). Analogously to Rickard's result, the cotilting derived equivalence can be represented by $T \otimes_R^\mathbf{L} -$ if $T$ is good. 

In Section 6, we obtain in \cref{thm-tensor} a commutative square for a good and decent tilting complex similar to \cref{I1}, which shows that the representable derived equivalences are compatible with the tensor and contratensor structures:
\begin{equation}\label{I2}
    \begin{tikzcd}
        \D^\bdd(\Mod R) \times \D^\bdd(\lMod R) \arrow{d}{\cong}[swap]{(\RHom_R(T,-), T \otimes_R^\mathbf{L} -)} \arrow{r}{- \otimes_R^\mathbf{L} -} & \D(\Mod \Zbb) \arrow{d}{=} \\
        \D^\bdd(\Ctra \Sfr) \times \D^\bdd(\Disc \Sfr) \arrow{r}{- \ctrtensor_{\Sfr}^\mathbf{L} -} & \D(\Mod \Zbb) 
    \end{tikzcd}
\end{equation}
Note that in this picture, certain asymmetry appears between the tilting and cotilting equivalence which was not visible in the classical tilting situation \cref{I1}. In our main \cref{converse}, we show that a converse of our results can be formulated as well, resulting in a generalization of Rickard's derived Morita theory for decent tilting complexes. Namely, there is a complete and separated topological ring $\Sfr$ and a couple of triangle equivalences $\D^\bdd(\Mod R) \cong \D^\bdd(\Ctra \Sfr)$ and $\D^\bdd(\lMod R) \cong \D^\bdd(\Disc \Sfr)$ which make the square \cref{I2} commute if and only if there is a decent tilting complex in $\D^\bdd(\Mod R)$ such that its endomorphism ring endowed with the compact topology is $\Sfr$.

In the final Section 7, we discuss an explicit application to tilting and cotilting complexes arising from a codimension function on the Zariski spectrum of a one-dimensional commutative noetherian ring. In particular, we describe the hearts as certain arrow categories, which resemble the constructions of torsion and complete models developed in \cite{Bsep} and \cite{torsionmodel}.
\section{$\Add$-closures in compactly generated triangulated categories}
The goal of this section is to partially extend the techniques of \cite[\S 6, \S7]{PS21} to a triangulated context by giving a description of the $\Add$-closure of an object $M$ using a topological structure of the endomorphism ring of $M$. 
\subsection{Contramodules over topological rings}
For a comprehensive resource about the theory of contramodules over complete and separated topological rings as developed by Positselski, we refer the reader to the survey \cite{contramodules} and references therein, as well as to \cite[\S 6, \S 7]{PS21} where the setting is very close to ours. Here we recall the basic concepts and notation. Let $\Rfr$ be a (unital, associative) ring. We say that $\Rfr$ is a \newterm{(left) topological ring} if it comes endowed with a linear topology of left ideals, that is, with a filter $(\Ifr_\alpha)_{\alpha \in A}$ of left ideals of $\Rfr$ such that for each $r \in \Rfr$ and $\alpha \in A$ there is $\alpha' \in A$ such that $\Ifr_{\alpha'} \cdot r \subseteq \Ifr_\alpha$. With such a filter fixed, we call the ideals it contains the \newterm{open} left ideals of the topological ring $\Rfr$. There is a natural map $\lambda: \Rfr \to \varprojlim_{\Ifr \subseteq \Rfr \text{ open}} \Rfr/\Ifr$ from $\Rfr$ to the completion with respect to the topology of open left ideals. We say that $\Rfr$ is \newterm{complete} if the map $\lambda$ is surjective and \newterm{separated} if the map is injective.

Let $\Rfr$ be a complete and separated topological ring. Given a set $X$, let $\Rfr[[X]]$\footnote{The notation ``$[[X]]\Rfr$'' would perhaps make better syntactic sense here. Since we are only dealing with right $\Rfr$-contramodules in this paper, we will stick to the symbol ``$\Rfr[[X]]$'' which is a bit easier to read.} denote the set of all (possibly infinite) formal linear combinations $\sum_{x \in X} x \cdot r_x$ of elements of the set $X$ with coefficients $r_x \in \Rfr$ such that the family $(r_x)_{x \in X}$ converges to zero in the topology of $\Rfr$, that is, if for any open left ideal $\Ifr$ we have $r_x \in \Ifr$ for all but finitely many $x \in X$. This assignment defines a functor $\Rfr[[-]]: \Sets \to \Sets$ on the category of all sets. Indeed, given a map $f: X \to Y$ of sets, the induced map $\Rfr[[f]]: \Rfr[[X]] \to \Rfr[[Y]]$ is defined by sending an element $\sum_{x \in X} x \cdot r_x$ to $\sum_{x \in X} f(x) \cdot r_x = \sum_{y \in Y}y \cdot s_y$, where the coefficient $s_y = \sum_{f(x) = y}r_x$ is well-defined using the fact that $\Rfr$ is complete and separated and the coefficients converge to zero. There is the ``opening of parentheses'' map $\mu_X: \Rfr[[\Rfr[[X]]]] \to \Rfr[[X]]$ (which is the obvious assignment of a formal linear combination to a formal linear combination of formal linear combinations) and the ``trivial linear combination'' map $\epsilon_X: X \to \Rfr[[X]]$, the well-defined-ness of $\mu_X$ is again ensured by the complete and separated assumption on the topology. Then the functor $\Rfr[[-]]$ on $\Sets$ together with the two natural transformation $\mu$ and $\epsilon$ form an additive monad on the category of sets, and one can therefore speak about modules (=algebras) over this monad --- these are precisely the \newterm{right $\Rfr$-contramodules}. Explicitly, a right $\Rfr$-contramodule is a set $\Mfr$ together with a \newterm{contraaction} map $\pi: \Rfr[[\Mfr]] \to \Mfr$ satisfying two axioms: first we have two maps $\mu_{\Mfr}, \Rfr[[\pi]]: \Rfr[[\Rfr[[\Mfr]]]] \to \Rfr[[\Mfr]]$ and these need to equalize after composing with $\pi: \Rfr[[\Mfr]] \to \Mfr$ (``contra-associativity''), and secondly the composition of $\pi \circ \epsilon_{\Mfr}: \Mfr \to \Rfr[[\Mfr]] \to \Mfr$ needs to be the identity map (``contra-unitality'').

We denote the category of all right $\Rfr$-contramodules by $\Ctra \Rfr$. It turns out that $\Ctra \Rfr$ is a complete and cocomplete locally presentable abelian category. For any set $X$, the map $\mu_X$ endows the set $\Rfr[[X]]$ with a natural structure of a right $\Rfr$-contramodule, and in fact, if $\star$ is a singleton set then $\Rfr = \Rfr[[\star]]$ is a projective generator of $\Ctra \Rfr$, while $\Rfr[[X]]$ is the coproduct of $X$ copies of $\Rfr$ in $\Ctra \Rfr$. In particular, $\Ctra \Rfr$ has enough projectives and the full subcategory of projective objects $\Ctrapr \Rfr$ consists precisely of direct summands of objects of the form $\Rfr[[X]]$ for some set $X$. There is a natural forgetful functor $\Ctra \Rfr \to \Mod \Rfr$ which simply restricts the contraaction just to finite linear combinations with coefficients in $\Rfr$. This forgetful functor is in general not fully faithful, and while it respects products, it usually does not preserve coproducts. 
\subsection{Modules over ringoids}
By a \newterm{ringoid} we understand a skeletally small preadditive category. A ringoid $\Rcal$ gives rise to the category $\Mod \Rcal$ of all contravariant additive functors $\Rcal \to \Mod \Zbb$, we call its objects the \newterm{right $\Rcal$-modules}. Then $\Mod \Rcal$ is a Grothendieck category with a generating set of finitely presented projective objects given by the representable functors $\Hom_\Rcal(-,R)$, where $R$ runs over some skeleton of $\Rcal$. The classical case of right modules over an associative unital ring is recovered by restricting to ringoids with precisely one object.

Given an object $X$ in an additive category $\Ccal$ with arbitrary coproducts, we let $\Add_{\Ccal}(X)$ denote the full subcategory consisting of direct summands of coproducts of copies of $X$, dually we also define $\Prod_{\Ccal}(X)$ using products. We drop the subscript if the ambient category $\Ccal$ is clear from the context. Given a ringoid $\Rcal$ and a module $M \in \Mod \Rcal$, consider the endomorphism ring $\Sfr = \End_{\Mod \Rcal}(M)$. Following \cite[\S 7.1]{PS21} and the references therein, we endow $\Sfr$ with the \newterm{finite topology}, in which the filter base of open left ideals consists of ideals of the form $\Ifr_F = \{g \in \End_{\Mod \Rcal}(M) \mid g_{\restriction F} = 0\}$ where $F$ runs through finitely generated subobjects $F$ of $M$. Equivalently, the base consists of ideals of the form $\Ifr_f = \{g \in \End_{\Mod \Rcal}(M) \mid g \circ f = 0\}$ where $f$ runs through morphisms $f: P \to M$ where $P$ is a finitely generated projective object of $\Mod \Rcal$. The following result of Positselski and Šťovíček is the initial point of our study.
\begin{theorem}\cite[Theorem 7.1, Proposition 7.3]{PS21}\label{T:PS}
    Let $\Rcal$ be a ringoid and $M \in \Mod \Rcal$. Consider the endomorphism ring $\Sfr = \End_{\Mod \Rcal}(M)$ endowed with the finite topology. Then $\Sfr$ is a complete and separated (left) topological ring and the functor $\Hom_{\Mod \Rcal}(M,-): \Mod \Rcal \to \Mod \Sfr$ factors through the forgetful functor $\Ctra \Sfr \to \Mod \Sfr$ and induces an equivalence $\Hom_{\Mod \Rcal}(M,-): \Add(M) \toeq \Ctrapr \Sfr$.
\end{theorem}
\begin{rmk}\label{R:PS}
    \cite[Theorem 7.1]{PS21} shows in particular that the two additive monads on the category of sets defined by the rules $X \mapsto \Hom_{\Mod \Rcal}(M,M^{(X)})$ and $X \mapsto \Sfr[[X]]$ (see \cite[\S 6.3]{PS21}) are isomorphic, which is how the natural $\Sfr$-module structure on $\Hom_{\Mod \Rcal}(M,M^{(X)})$ extends to a right $\Sfr$-contramodule structure.

    If $M \in \Mod \Rcal$ is a finitely generated module then the finite topology on $\Sfr$ becomes discrete, and thus $\Ctra \Sfr = \Mod \Sfr$. In this case, \cref{T:PS} recovers the classical equivalence $\Add(M) \cong \Modpr \Sfr$ of \cite{Dress}.
\end{rmk}
\subsection{Compactly generated triangulated categories}
Let $\Tcal$ be a triangulated category with suspension functor $-[1]$ and assume that $\Tcal$ has arbitrary coproducts. Recall that an object $F \in \Tcal$ is \newterm{compact} if the functor $\Hom_\Tcal(F,-): \Tcal \to \Mod \Zbb$ preserves coproducts, and denote by $\Tcal^\cpt$ the full subcategory of compact objects. Unless specified otherwise, we assume that the triangulated category $\Tcal$ is \newterm{compactly generated}, which means that $\Tcal^\cpt$ is skeletally small and any object $X \in \Tcal$ such that $\Hom_\Tcal(F,X) = 0$ for all $F \in \Tcal^\cpt$ has to be zero; this assumption also implies that $\Tcal$ has arbitrary products. 

Viewing $\Tcal^\cpt$ as a ringoid yields a theory of purity in $\Tcal$, as developed by Krause \cite{Kr00}. The main ingredient is the \newterm{restricted Yoneda functor} $\yo: \Tcal \rightarrow \Mod \Tcal^\cpt$ defined by the assignment $X \mapsto \Hom_{\Tcal}(-,X)_{\restriction \Tcal^\cpt}$. This is a conservative cohomological functor (however, it is rarely faithful and can even fail to be full). A triangle $X \xrightarrow{f} Y \xrightarrow{g} Z \xrightarrow{} X[1]$ in $\Tcal$ is \newterm{pure} provided that it is sent to a short exact sequence $0 \to \yo X \xrightarrow{\yo f} \yo Y \xrightarrow{\yo g} \yo Z \to 0$ in $\Mod \Tcal^\cpt$. If this is the case, we say that $f$ (resp. $g$) is a \newterm{pure monomorphism} (resp. \newterm{pure epimorphism}) in $\Tcal$. An object $E \in \Tcal$ is \newterm{pure-injective} provided that $\yo E$ is an injective object in $\Mod \Tcal^\cpt$. A \newterm{pure-injective envelope} of object $X$ is an object $PE(X)$ together with a map $e: X \to PE(X)$ such that $\yo e: \yo X \rightarrow \yo PE(X)$ identifies with the injective envelope map in $\Mod \Tcal^\cpt$. The pure-injective envelope exists for any object $X \in \Tcal$ and $PE(X)$ is uniquely determined up to isomorphism. \newterm{Pure-projective} objects in $\Tcal$ are defined analogously. A morphism $f$ in $\Tcal$ is called \newterm{phantom} provided that $\yo f = 0$. For $X,Y \in \Tcal$, the kernel of the map $\Hom_{\Tcal}(X,Y) \rightarrow \Hom_{\Mod \Tcal^\cpt}(\yo X, \yo Y)$ consists precisely of the phantom morphisms $X \to Y$. If $X$ is pure-projective (in particular, if $X$ is compact) or if $Y$ is pure-injective, any phantom map $X \to Y$ is zero in $\Tcal$. See \cite[\S 1, \S 2]{Kr00} for details.

Let $M \in \Tcal$ and denote its endomorphism ring by $\Sfr = \End_{\Tcal}(M)$. In complete analogy with the finite topology, we define a natural linear topology on the ring $\Sfr$ which we call the \newterm{compact topology}. The base of open left ideals of $\Sfr$ is given by ideals of the form $\Ifr_f = \{g \in \Sfr \mid g \circ f = 0\}$ for all maps $f: F \rightarrow M$ in $\Tcal$ with $F \in \Tcal^\cpt$. The compact topology makes $\Sfr$ into a left topological ring. Indeed, for any $f: F \to M$ with $F \in \Tcal^\cpt$ and any $s \in \Sfr$ we have $\Ifr_{s \circ f} \cdot s = \{g \circ  s \in \Sfr \mid g \circ  s \circ  f = 0\} \subseteq \Ifr_f = \{ g \in \Sfr \mid g \circ  f = 0\}$.

\begin{rmk}
    Let $\Rcal$ be a ringoid, $M \in \Mod \Rcal$, and let $\Sfr = \End_{\Mod \Rcal}(M) = \End_{\D(\Mod \Rcal)}(M)$ be the endomorphism ring; here $\D(\Mod \Rcal)$ is the derived category of $\Mod \Rcal$. Then $\Sfr$ can be endowed with either the finite topology induced in $\Mod \Rcal$ or the compact topology induced in $\D(\Mod \Rcal)$, and these two topologies coincide. Indeed, the compact topology is clearly finer because finitely generated projectives modules are compact objects of the derived category. On the other hand, consider $C \in \D(\Mod \Rcal)^\cpt$ a compact object. It is well-known that $C$ can be assumed to be a bounded complex of finitely generated projective $\Rcal$-modules. Then we have $\Hom_{\D(\Mod \Rcal)}(C,M) \cong \Hom_{\K(\Mod \Rcal)}(C,M) \cong \Hom_{\K(\Mod \Rcal)}(\sigma^{\leq 0}C,M) \cong \Hom_{\D(\Mod \Rcal)}(\tau^{\geq 0}\sigma^{\leq 0}C,M)$, where $\sigma^{\leq 0}$ is the stupid truncation to non-positive degrees, $\tau^{\geq 0}$ the smart truncation to non-negative degrees, and $\K(\Mod \Rcal)$ the homotopy category. Note that the induced map $\sigma^{\geq 0}\sigma^{\leq 0}C \to \tau^{\geq 0}\sigma^{\leq 0}C$ is an epimorphism in $\Mod \Rcal$. Then for any morphism $f: C \to M$ we have $\Ifr_f = \{ g \in \Sfr \mid g \circ f = 0\} = \{ g \in \Sfr \mid g \circ (\tau^{\geq 0}\sigma^{\leq 0}f) = 0\} = \{ g \in \Sfr \mid g \circ (\sigma^{\geq 0}\sigma^{\leq 0}f) = 0\} = \Ifr_{\sigma^{\geq 0}\sigma^{\leq 0}f}$, where $\sigma^{\geq 0}\sigma^{\leq 0}f$ is an $\Rcal$-module map from the finitely generated projective $\sigma^{\geq 0}\sigma^{\leq 0}C$ to $M$. This shows that the finite topology refines the compact topology, and so the two topologies coincide.
\end{rmk}

Unlike in the module-theoretic setting of \cref{T:PS}, we cannot expect the compact topology to be always separated for an object $M \in \Tcal$. Indeed, $\End_\Tcal(M)$ can easily contain non-zero phantom maps. Therefore, we are forced to impose a purity condition to obtain the following triangulated version of \cref{T:PS}.

\begin{prop}\label{addM-tria}
    Let $M \in \Tcal$ be such that the restriction $\yo_{\restriction \Add(M)}: \Add(M) \rightarrow \Mod \Tcal^\cpt$ of the restricted Yoneda functor to $\Add(M)$ is fully faithful. Then $\Sfr=\End_\Tcal(M)$ endowed with the compact topology is a complete and separated topological ring, and there is an equivalence $\Hom_{\Tcal}(M,-): \Add(M) \toeq \Ctrapr \Sfr$.
\end{prop}
\begin{proof}
    The assumption yields in particular that $\yo$ induces a ring isomorphism $\Sfr \cong \End_{\Mod \Tcal^\cpt}(\yo M)$. Since $\yo$ preserves coproducts, we also have that $\Add(M) = \Add_{\Tcal}(M) \cong \Add_{\Mod \Tcal^\cpt}(\yo M)$. Then it follows directly from \cref{T:PS} that $\Sfr$ is complete and separated and $\Add(M)$ is equivalent to $\Ctrapr \Sfr$, with the caveat that here $\Sfr$ is endowed with the finite topology induced on $\yo M \in \Mod \Tcal^\cpt$. The equivalence is induced by the functor $\Hom_{\Mod \Tcal^\cpt}(\yo M,-): \Mod \Tcal^\cpt \to \Ctra \Sfr$, and the restriction of this functor to $\Add_{\Mod \Tcal^\cpt}(\yo M)$ is identified with the functor $\Hom_{\Tcal}(M,-): \Add(M) \rightarrow \Ctrapr \Sfr$.
    
    It remains to show that the compact topology induced in $\Tcal$ and the finite topology induced in $\Mod \Tcal^\cpt$ on $\Sfr$ coincide. But to see that, it is enough to recall that finitely generated subobjects $F \xhookrightarrow{} \yo M$ are precisely the images of morphisms $\yo C \rightarrow \yo M$ with $C \in \Tcal^\cpt$, because the objects of the form $\yo C$ with $C \in \Tcal^\cpt$ are precisely the finitely generated projective objects of $\Mod \Tcal^\cpt$. But any such morphism is of the form $\yo(f)$ for a morphism $f: C \rightarrow M$ in $\Tcal$, see \cite[Theorem 1.8]{Kr00}. It follows that the ring isomorphism $\Sfr \cong \End_{\Mod \Tcal^\cpt}(\yo M)$ induced by $\yo$ identifies the compact and the finite topology.
\end{proof}
\begin{cor}
    Let $M$ be a pure-projective object of $\Tcal$. Then $\Add(M) \cong \Ctrapr \Sfr$. In particular, if $M$ is a compact object then $\Add(M) \cong \Ctrapr \Sfr = \Modpr \Sfr$.
\end{cor}
\begin{proof}
    If $M$ is pure-projective then the isomorphism follows directly from \cref{addM-tria} and \cite[Theorem 1.8]{Kr00}. If $M$ is compact then $M$ is pure-projective and, similarly to \cref{R:PS}, the compact topology on $\Sfr$ becomes discrete so that $\Ctra \Sfr = \Mod \Sfr$.
\end{proof}
\subsection{Hearts of {{\tst}}structures}
For a short moment, let $\Tcal$ be \newterm{any} triangulated category. Given a full subcategory $\Ccal \subseteq \Tcal$, we define full subcategories $\Ccal\Perp{\bigcirc} = \{X \in \Tcal \mid \Hom_\Tcal(C,X[i])=0 ~\forall C \in \Ccal, i ~\bigcirc\}$ of $\Tcal$ where $\bigcirc$ denotes a symbol such as $0,\leq 0, >0, \neq 0$ and others which specify a set of integers. Analogously, we define full subcategories $\Perp{\bigcirc}\Ccal = \{X \in \Tcal \mid \Hom_\Tcal(X,C[i])=0 ~\forall C \in \Ccal, i ~\bigcirc\}$ and if $\Ccal = \{C\}$, we write just $C\Perp{\bigcirc}$ and $\Perp{\bigcirc}C$. Recall that a \newterm{{{\tst}}structure} $\Tbb$ in $\Tcal$ is a pair $\Tbb = (\Ucal,\Vcal)$ of full subcategories such that:
\begin{enumerate}
    \item[(i)] $\Vcal = \Ucal\Perp{0}$,
    \item[(ii)] $\Ucal$ is closed under suspension, that is, $\Ucal[1] \subseteq \Ucal$,
    \item[(iii)] each object $X \in \Tcal$ fits into a triangle $U \to X \to V \to U[1]$ with $U \in \Ucal$ and $V \in \Vcal$.  
\end{enumerate}
Following Beilinson, Bernstein, and Deligne \cite{BBD81}, any {{\tst}}structure gives rise to an abelian category $\Hcal_\Tbb = \Ucal \cap \Vcal[1]$ called the \newterm{heart} of the {{\tst}}structure. The short exact sequences in $\Hcal_\Tbb$ are induced by those triangles of $\Tcal$ with all three components lying in the heart. In particular, a morphism $f: X \to Y$ in $\Hcal$ is a monomorphism if and only if its mapping cone $\Cone(f)$ in $\Tcal$ belongs to $\Hcal$.

Let now $\Tcal$ be again a compactly generated triangulated category, and let us recall an important notion from the theory of purity. A subcategory $\Ccal$ of $\Tcal$ is called \newterm{definable} provided that there is a set $\Phi$ of maps in $\Tcal^\cpt$ such that 
$$\Ccal = \{X \in \Tcal \mid \Hom_\Tcal(f,X) \text{ is the zero map $~\forall f \in \Phi$}\}.$$ 
Given an object $X \in \Tcal$, let $\Def_\Tcal(X)$ (or just $\Def(X)$ if $\Tcal$ is clear from context) denote the smallest definable subcategory of $\Tcal$ which contains $X$. Note that $\Def(X)$ exists, as it can be defined using the set $\Phi$ of all maps $f$ in $\Tcal^\cpt$ such that $\Hom_\Tcal(f,X)$ is the zero map. It is clear from the definition that any definable subcategory is closed under coproducts, products, pure subobjects, and pure quotients. Here, an object $Y \in \Tcal$ is a \newterm{pure subobject} of $X \in \Tcal$ if there is a pure monomorphism $Y \to X$, the \newterm{pure quotients} are defined similarly. By \cite[Corollary 4.4]{AHMV17}, definable subcategories are also closed under taking pure-injective envelopes, meaning that $PE(X)$ belongs to a definable subcategory whenever $X$ does.

\begin{prop}\label{tstr-gen}
Let $\Tbb$ be a {{\tst}}structure in $\Tcal$ and consider an object $M \in \Hcal_\Tbb$ in its heart. Assume the following condition:  $\Def_{\Tcal}(M) \subseteq \Hcal_\Tbb$. 
    
Then the restriction $\yo_{\restriction \Def_{\Tcal}(M)}: \Def_{\Tcal}(M) \to \Mod \Tcal^\cpt$ is fully faithful. As a consequence, $\Add_\Tcal(M) \cong \Ctrapr \Sfr$ via $\Hom_\Tcal(M,-)$, where $\Sfr = \End_{\Tcal}(M)$ is endowed with the compact topology, and this topological ring is complete and separated. 
\end{prop} 
\begin{proof}
    Recall from the discussion above that $\Def_{\Tcal}(M)$ is closed under pure subobjects, pure quotients, and pure-injective envelopes. Then the condition $\Def_{\Tcal}(M) \subseteq \Hcal_\Tbb$ implies that for any $Y \in \Def_{\Tcal}(M)$, all three components of the triangle $Y \to PE(Y) \to L \xrightarrow{+}$ belong to $\Def_{\Tcal}(M)$, and thus also to $\Hcal_\Tbb$. It follows that the pure-injective envelope map $i: Y \to PE(Y)$ is a monomorphism in the heart $\Hcal_\Tbb$. If $X \in \Hcal_\Tbb$ and $f: X \to Y$ is a phantom map in $\Tcal$, then $i \circ f$ is zero (both in $\Tcal$ and $\Hcal_\Tbb$). Since $i$ is a monomorphism in $\Hcal_\Tbb$, it follows that $f  = 0$. We showed that $\yo: \Hom_{\Tcal}(X,Y) \to \Hom_{\Mod \Tcal^\cpt}(\yo X, \yo Y)$ is a monomorphism for any $X \in \Hcal_\Tbb$ (in particular, for any $X \in \Def_{\Tcal}(M)$) and $Y \in \Def_{\Tcal}(M)$.

    Furthermore, $\yo$ induces the following commutative diagram with exact rows:
    $$
    \begin{tikzcd}[column sep=tiny, row sep=normal]
        \Hom_{\Tcal}(X,Y[-1]) \arrow{r}\arrow[two heads]{d} & \Hom_{\Tcal}(X,Y) \arrow{r}\arrow[hook]{d} & \Hom_{\Tcal}(X,PE(Y)) \arrow{r}\arrow[hook,two heads]{d} & \Hom_{\Tcal}(X,L) \arrow[hook]{d}\\
        0 \arrow{r} & \Hom_{\Mod \Tcal^\cpt}(\yo X,\yo Y) \arrow{r} & \Hom_{\Mod \Tcal^\cpt}(\yo X,\yo PE(Y)) \arrow{r} & \Hom_{\Mod \Tcal^\cpt}(\yo X,\yo L) 
    \end{tikzcd}$$
    The bottom row is exact because $Y \to PE(Y) \to L \xrightarrow{+}$ is a pure-triangle, and therefore it induces a short exact sequence $0 \to \yo Y \to \yo PE(Y) \to \yo L \to 0$.
    The three rightmost vertical arrows are monomorphisms by the previous discussion, while the third arrow from the left is an isomorphism since $PE(Y)$ is pure-injective \cite[Theorem 1.8]{Kr00}. It follows by the Four Lemma that the second vertical arrow from the left is an isomorphism and we are done.

    Finally, since $\Add_{\Tcal}(M) \subseteq \Def_{\Tcal}(M)$, we get the last sentence of the statement from \cref{addM-tria}.
\end{proof}
\begin{prop}\label{heartequiv}
    Let $\Tbb$ be a {{\tst}}structure in $\Tcal$ and assume there is a projective generator $P$ of the heart $\Hcal_\Tbb$. Assume further that $\Def_{\Tcal}(P) \subseteq \Hcal_\Tbb$ is satisfied and let $\Sfr = \End_{\Tcal}(P)$ be endowed with the compact topology. Then $\Sfr$  is complete and separated and there is an equivalence of abelian categories $\Hom_{\Hcal_\Tbb}(P,-): \Hcal_\Tbb \toeq \Ctra \Sfr$.
\end{prop}
\begin{proof}
    By \cref{tstr-gen}, $\Hom_{\Tcal}(P,-)$ restricts to an equivalence $\Add_\Tcal(P) \toeq \Ctrapr \Sfr$. Note that the assumption on $P$ ensures that $\Add_{\Tcal}(P) = \Add_{\Hcal_\Tbb}(P) = (\Hcal_\Tbb)_{\mathsf{proj}}$. 
    
    Clearly, $\Hom_{\Hcal_\Tbb}(P,-)$ is well-defined as a functor $\Hcal_\Tbb \to \Mod \Sfr$. Following \cite[Corollary 6.3]{PS21} (also in view of \cite[Remark 6.4 and \S 6.3]{PS21}), this functor factorizes as $\Hom_{\Hcal_\Tbb}(P,-): \Hcal_\Tbb \to \Ctra \Sfr \to \Mod \Sfr$, where the latter is the forgetful functor. Since $\Hom_{\Tcal}(P,-) = \Hom_{\Hcal_\Tbb}(P,-): \Hcal_\Tbb \to \Ctra \Sfr$ is an exact functor between two abelian categories with enough projectives which restricts to an equivalence between the respective categories of projectives, a standard argument yields that we have the desired equivalence.
\end{proof}

\subsection{Silting objects}
In what follows we will focus on {{\tst}}structures coming from the theory of (large) silting and cosilting objects, we refer the reader to the survey \cite[\S 5, \S 6]{AH19} as the basic resource. An object $T$ of $\Tcal$ is \newterm{silting} if the pair $(T\Perp{>0},T\Perp{\leq 0})$ constitutes a {{\tst}}structure in $\Tcal$, which we then call the \newterm{silting {{\tst}}structure}. Any object $T' \in \Tcal$ satisfying $\Add(T) = \Add(T')$ is also a silting object inducing the same silting {{\tst}}structure. In this situation, we say that $T$ and $T'$ are \newterm{equivalent} silting objects. Equivalently, two silting objects $T$ and $T'$ are equivalent if they induce the same silting t-structure, see \cite[Lemma 4.5]{PV20}. Denote the associated heart by $\Hcal_T$ and the induced cohomological functor as $H_T^0: \Tcal \to \Hcal_T$. Then $\Hcal_T$ is an abelian category with exact products such that $H_T^0(T)$ is a projective generator (\cite[Lemma 2.7]{AHMV17}). Note also that in this situation we have $\Hcal_T = T\Perp{\neq 0}$. If in addition $\Add(T) \subseteq T\Perp{<0}$, or equivalently $\Add(T) \subseteq \Hcal_T$, we call $T$ a \newterm{tilting} object. Since then $T = H^0_T(T)$, a tilting object $T$ identifies itself with a projective generator of $\Hcal_T$.

Now we are ready to state the key definition of this paper.

\begin{dfn}
    A silting object $T \in \Tcal$ will be called \newterm{decent} provided that $\Def(T) \subseteq \Hcal_T$. 
\end{dfn}
\begin{rmk}    
    Note that since $\Add(T) \subseteq \Def(T)$, any decent silting object $T$ is a tilting object. Also, clearly the property of being decent is invariant under equivalence of silting objects, and therefore can be viewed as a property of the silting {{\tst}}structure. By \cref{heartequiv}, the heart $\Hcal_T$ of any decent tilting object is equivalent to $\Ctra \Sfr$, where $\Sfr = \End_\Tcal(T)$ is endowed with the compact topology.
\end{rmk}

The remaining results of this section need to assume that $\Tcal$ has a model. For our purposes, a rather weak enhancement in the form of a Grothendieck derivator will suffice, which in particular allows to define the notion of a directed homotopy colimit and homotopy coherent reduced product; we refer to \cite[\S 2, Appendix]{Lak20} and references therein. We remark that $\Tcal$ underlies a derivator whenever $\Tcal$ is the homotopy category of a model category or an $\infty$-category. For what comes in the next sections, we recall that the unbounded derived category $\D(\Mod R)$ of a module category over a ring $R$ underlies the standard compactly generated derivator induced by any of the standard model structures on the category of cochain complexes, the directed homotopy colimits are in this case computed as ordinary direct limits of diagrams in the category of cochain complexes of $R$-modules, see e.g. \cite[Appendix]{HN21}.

\begin{theorem}\label{gen-der-eq}
    Let $\Tcal$ be a triangulated category which is the base of a compactly generated derivator and $T \in \Tcal$ a decent tilting object. Define the following triangulated subcategory of $\Tcal$ of objects bounded with respect to the silting t-structure induced by $T$: 
    $$\Tcal^\bdd = \{X \in \Tcal \mid \Hom_\Tcal(T,X[i]) = 0 \text{ for all but finitely many $i \in \Zbb$}\}.$$ 
    Then there is a triangle equivalence $T^\bdd \cong \D^\bdd(\Ctra \Sfr)$, where $\Sfr = \End_\Tcal(T)$ is endowed with the compact topology.
\end{theorem}
\begin{proof}
    By Modoi's result \cite[Theorem 7]{Modoi}, $\Tcal$ admits an f-enhancement in the sense of \cite{PV20}, see \cref{ss-derequivalence} for more context. It follows by \cite[Theorem 3.11, Proposition 5.1]{PV20} that the realization functor $\D^\bdd(\Hcal_T) \to \Tcal$ exists and it is fully faithful. As the essential image of this functor is precisely $\Tcal^\bdd$ by \cite[Theorem 3.11]{PV20} and $\Hcal_T \cong \Ctra \Sfr$ by \cref{ab-eq}, we are done.
\end{proof}

For the next part it will be important to use the fact that definable subcategories are characterized by their closure properties in case $\Tcal$ underlies a stable derivator. This was developed by Laking \cite{Lak20}. 

\begin{theorem}[\cite{Lak20}]\label{T:LV}
    Let $\Tcal$ be a triangulated category which is the base of a compactly generated derivator. Then:
    \begin{enumerate}
        \item A full subcategory $\Ccal$ of $\Tcal$ is definable if and only if it is closed under products, pure subobjects, and directed homotopy colimits.
        \item For any $Y \in \Tcal$, the definable closure $\Def_\Tcal(Y)$ consists precisely of pure subobjects of directed homotopy colimits of objects from $\Prod_\Tcal(Y)$.
    \end{enumerate}
\end{theorem}
\begin{proof}
Both the claims are proved in \cite[Theorem 3.11]{Lak20} and \cite[Corollary 3.12]{Lak20}. 
\end{proof}
\begin{lemma}\label{def-tilting}
    Let $\Tcal$ be a triangulated category which is the base of a compactly generated derivator and $T \in \Tcal$ a pure-projective tilting object. Then $\Hcal_T$ is a definable subcategory of $\Tcal$. In particular, $T$ is decent.
\end{lemma}
\begin{proof}
    Since $T$ is pure-projective, both the subcategories $T\Perp{>0}$ and $T\Perp{<0}$ are closed under pure monomorphisms, pure epimorphisms, and products. It follows that both the subcategories are closed under directed homotopy colimits and so are definable, this follows by a similar argument as in the proof of \cite[Lemma 4.5]{Lak20}. Then their intersection $\Hcal_T$ is definable as well, and the condition $\Def(T) \subseteq \Hcal_T$ clearly holds, so that $T$ is decent by definition.
\end{proof}
An object $C$ of $\Tcal$ is \newterm{cosilting} if the pair $(\Perp{\leq 0}C,\Perp{> 0}C)$ constitutes a {{\tst}}structure in $\Tcal$. In order to adhere to the standard notation, it is convenient to consider the associated heart shifted by degree $-1$ so that the equality $\Hcal_C = \Perp{\neq 0}C$ holds. We denote the induced (again, shifted) cohomological functor as $H_C^0: \Tcal \to \Hcal_C$. The heart $\Hcal_C$ is an abelian category with exact coproducts such that $H_C^0(C)$ is an injective cogenerator (\cite[Lemma 2.7]{AHMV17}). If in addition $\Prod(C) \subseteq \Perp{<0}C$, or equivalently $\Prod(C) \subseteq \Hcal_C$, we call $C$ a \newterm{cotilting} object. Similarly to the silting situation, we say that two cosilting objects $C$ and $C'$ are \newterm{equivalent} if they induced the same {{\tst}}structure, which amounts to $\Prod(C) = \Prod(C')$.

\begin{lemma}\label{def-cotilting}
    Let $\Tcal$ be a triangulated category which is the base of a compactly generated derivator and $C \in \Tcal$ a pure-injective cotilting object. Then $\Def(C) \subseteq \Hcal_C$.
\end{lemma}
\begin{proof}
    Since $C$ is pure-injective, both the categories $\Perp{>0}C$ and $\Perp{<0}C$ are closed under pure monomorphisms, pure epimorphisms, and coproducts. It follows that their intersection $\Hcal_C$ is closed under pure monomorphisms and directed homotopy colimits (see the proof of \cite[Lemma 4.5]{Lak20}). Since $C$ is cotilting, we have $\Prod(C) \subseteq \Hcal_C$. Now it is enough to recall from \cref{T:LV}(2) that $\Def(C)$ consists precisely of all pure subobjects of directed homotopy colimits of objects in $\Prod(C)$. 
\end{proof}
\begin{rmk}
    Certain asymmetry is now apparent from \cref{def-tilting,def-cotilting}. First, unlike for pure-projective tilting objects, it is often not the case that the heart $\Hcal_C$ of a pure-injective cotilting object is a definable subcategory. 
    
    Secondly, one encounters far fewer pure-projective tilting objects in practice than the pure-injective cotilting counterparts. Of course, any compact tilting object is pure-projective. There exist pure-projective tilting complexes which are not equivalent to a compact one but these are rather exotic and hard to construct, see \cite{PPT}. On the other hand, any (bounded) cosilting complex in a derived category of a ring is known to be pure-injective \cite[Proposition 3.10]{MV18}. In fact, the author is not aware of any example of a cosilting object of a compactly generated triangulated category which is not pure-injective.

    In what follows, we will show that the condition of being decent for a tilting complex $T$ in a derived category of a ring holds much more generally then under the pure-projective assumption.
\end{rmk}
\section{Tilting complexes and hearts in the derived category of a ring}
If $\Acal$ is an abelian category, we let $\D(\Acal)$ denote the (unbounded) derived category and $\D^\bdd(\Acal)$ the bounded derived category of cochain complexes over $\Acal$. The abelian categories we consider never encounter any set-theoretic issues with the existence of their derived categories.

Unless said otherwise, $R$ will always denote an arbitrary (associative, unital) ring. An object $T \in \D(\Mod R)$ of the derived category of right $R$-modules is a (bounded) \newterm{silting complex}\footnote{Silting complexes not necessarily isomorphic to a bounded complex of projectives are also considered in the literature. In this paper, we focus on the bounded versions of silting and cosilting complexes.} if $T$ is a silting object which is in addition quasi-isomorphic to a bounded complex of right projective $R$-modules. A silting complex $T$ is called a \newterm{tilting complex} in case it is a tilting object in $\D(\Mod R)$. Finally, a silting complex is \newterm{decent} if it is decent as a silting object, that is, if $\Def(T) = \Def_{\D(\Mod R)}(T) \subseteq \Hcal_T$. We recall that any decent silting complex is automatically tilting. Note that since $T\Perp{>0}$ is always a definable subcategory of $\D(\Mod R)$ \cite[Theorem 3.6]{MV18}, a silting complex $T$ is decent if and only if $\Def(T) \subseteq T\Perp{<0}$.

Let $\lMod R$ denote the category of left $R$-modules. Dually to the previous paragraph, an object $C \in \D(\lMod R)$ of the derived category of left $R$-modules is a \newterm{cosilting complex} if $C$ is isomorphic in $\D(\lMod R)$ to a bounded complex of injective $R$-modules and it is a cosilting object in this triangulated category. Any (bounded) cosilting complex is pure-injective in $\D(\lMod R)$ and $\Hcal_C$ is a Grothendieck category with an injective cogenerator $H_C^0(C)$ \cite[Proposition 3.10]{MV18}, \cite[Lemma 2.7, Theorem 3.6]{AHMV17}. If $C$ is in addition a cotilting object in $\D(\lMod R)$, we call it a \newterm{cotilting complex}. Finally, we say that a cosilting complex is of \newterm{cofinite type} if the associated cosilting {{\tst}}structure is compactly generated (see \cite[\S 2.6]{AHH19} for details).

Consider the character duality functor $(-)^+ := \Hom_{\Zbb}(-,\Qbb/\Zbb)$ acting as a functor $(\Mod R)^\op \to \lMod R$. Since $(-)^+$ is exact, it easily extends to a functor $(-)^+: \D(\Mod R)^\op \to \D(\lMod R)$. Note that this is a conservative functor, a property we will utilize throughout the rest of the paper. We will also need another duality functor $(-)^* := \RHom_{R}(-,R)$ which induces an equivalence $(\D(\lMod R)^\cpt)^\op \toeq \D(\Mod R)^\cpt$ on the categories of compact objects. See \cite[\S 2.2]{AHH19} for details about these dualities, and also the recent preprint \cite{BW} where the triangulated character duality is developed in larger generality. By symmetry, we also freely consider both functors going in the opposite direction.

The cosilting complexes of cofinite type are known to be precisely the character duals of silting complexes, up to equivalence.
\begin{prop}\cite[Theorem 3.3]{AHH19}\label{silting-cosilting}
    If $T \in \D(\Mod R)$ is a silting complex then $T^+ \in \D(\lMod R)$ is a cosilting complex. This assignment yields a bijection between equivalence classes of silting complexes in $\D(\Mod R)$ and equivalence classes of cosilting complexes in $\D(\lMod R)$ of cofinite type.
\end{prop}
The following lemma shows that applying the character duality induces a pair of definable categories which are in the usual terminology from the model theory of modules called \newterm{dual definable}.
\begin{lemma}\label{dual-definable}
    Let $Y \in \D(\Mod R)$. Then:
    \begin{enumerate}
        \item[(i)] For any $X \in \D(\Mod R)$ we have $X \in \Def_{\D(\Mod R)}(Y) \iff X^+ \in \Def_{\D(\lMod R)}(Y^+)$,
        \item[(ii)] For any $Z \in \D(\lMod R)$ we have $Z \in \Def_{\D(\lMod R)}(Y^+) \iff Z^+ \in \Def_{\D(\Mod R)}(Y)$.
    \end{enumerate}
\end{lemma}
\begin{proof}
    Let $\Phi$ be the set of all maps $f$ in $\D(\Mod R)^\cpt$ such that $\Hom_{\D(\Mod R)}(f,Y)$ is zero so that $\Def(Y) = \{X \in \D(\Mod R) \mid \Hom_{\D(\Mod R)}(f,X) \text{ is zero $~\forall f \in \Phi$}\}$. Put $\Phi^* = \{f^* \mid f \in \Phi\}$. We recall that $f^* = \RHom_{R}(f,R)$ and $\Phi^*$ is a set of maps between compact objects of $\D(\lMod R)$. For any map $f$ in $\D(\Mod R)^\cpt$ and any $X \in \D(\Mod R)$ we have an equivalence 
    $$\Hom_{\D(\Mod R)}(f,X) \text{ is zero}  \iff \Hom_{\D(\lMod R)}(f^*,X^+) \text{ is zero},$$ 
    see the proof of \cite[Lemma 2.3]{AHH19}. Plugging $X = Y$, we obtain that the set $\Phi^*$ of maps in $\D(\lMod R)^\cpt$ defines $\Def(Y^+)$, and then using it for general $X \in \D(\Mod R)$ we obtain $(i)$. Since $\Phi^{**}$ identifies with $\Phi$, we also obtain $(ii)$ by the same argument.
\end{proof}
\begin{lemma}\label{heartduality}
    Let $T$ be a silting complex in $\D(\Mod R)$ and let $C = T^+$ be the cosilting complex in $\D(\lMod R)$ corresponding to $T$ via \cref{silting-cosilting}. For any $X \in \D(\lMod R)$ we have $X \in \Hcal_C$ if and only if $X^+ \in \Hcal_T$.
\end{lemma}
\begin{proof}
    Note that $\Hcal_T = T\Perp{\neq 0}$ and our convention also ensures that $\Hcal_C = \Perp{\neq 0}C$. Using $\otimes$-$\Hom$ adjunction, there are isomorphisms for any $X \in \D(\lMod R)$ and $i \in \Zbb$
    $$\Hom_{\D(\lMod R)}(X,T^+[i]) \cong H^0(T \otimes_R^\mathbf{L} X[-i])^+ \cong \Hom_{\D(\Mod R)}(T,X^+[i]),$$
    which is enough to conclude the proof.
\end{proof}
We continue with the main result of this section which provides a more intuitive and useful characterization of the decent property of a silting complex in $\D(\Mod R)$.
\begin{theorem}\label{condition}
    Let $T \in \D(\Mod R)$ be a silting complex and put $C = T^+$. The following are equivalent:
    \begin{enumerate}
        \item[(i)] $T$ is decent,
        \item[(ii)] the cosilting complex $C$ is cotilting.
    \end{enumerate}
\end{theorem}
\begin{proof}
    $(i) \implies (ii)$: Since $\Def_{\D(\Mod R)}(T)$ is closed under taking the double character dual $(-)^{++}$ (this follows e.g. by \cref{dual-definable}), condition $(i)$ ensures that for any $X \in \Add(T)$ we have $X^{++} \in \Hcal_T$. In turn, we have $X^+ \in \Hcal_C$ by \cref{heartduality}. This already shows that $C^\varkappa \in \Hcal_C$ for any cardinal $\varkappa$. Since $\Hcal_C$ is closed under direct summands, we have $\Prod(C) \subseteq \Hcal_C$ as desired.

    $(ii) \implies (i)$: By \cref{dual-definable}, the definable closures of $T$ and $C = T^+$ satisfy that $X \in \Def_{\D(\Mod R)}(T)$ if and only if $X^+ \in \Def_{\D(\lMod R)}(C)$. By \cref{heartduality,def-cotilting}, we have for any $X \in \Def_{\D(\Mod R)}(T)$ that $X^{++} \in \Hcal_T$. As a consequence, any pure-injective object from $\Def_{\D(\Mod R)}(T)$ belongs to $\Hcal_T$, see \cite[Corollary 2.8]{AHH19}.

    Now let $X \in \Def_{\D(\Mod R)}(T)$ be any object. As recalled above, $T\Perp{>0}$ is definable, and so it is enough to show $X \in T\Perp{<0}$ to demonstrate that $X \in \Hcal_T$. Inductively, we can construct a sequence of triangles $X_i \to P_i \to X_{i+1} \to X_i[1]$ indexed by $i \geq 0$ with the following properties: The map $X_i \to P_i$ is the pure-injective envelope for all $i \geq 0$ and $X_0 = X$. It follows that $X_i, P_i \in \Def_{\D(\Mod R)}(T)$ for all $i \geq 0$. By the previous paragraph, $P_i \in \Hcal_T$ for all $i \geq 0$. Therefore, we have for any $l>0$ an isomorphism $\Hom_{\D(\Mod R)}(T,X_0[-l]) \cong \Hom_{\D(\Mod R)}(T,X_i[-l-i])$. 
    
    For any set of integers $I$, the full subcategory of $\D(\Mod R)$ determined by vanishing of cohomology in degrees outside of $I$ is easily seen to be definable by the set of identity maps $\Phi = \{R[-i] \xrightarrow{=} R[-i] \mid i \in \Zbb \setminus I\}$. Therefore, since $T$ is a cohomologically bounded object, there are integers $a<b$ such that any object from $\Def_{\D(\Mod R)}(T)$ has cohomology vanishing outside of degrees in the interval $[a,b]$. But then $\Hom_{\D(\Mod R)}(T,X_i[-l-i])$ has to vanish whenever $i \geq 0$ is chosen large enough. Using the above isomorphism, we infer that $\Hom_{\D(\Mod R)}(T,X_0[-l]) = 0$ for all $l>0$, meaning that $X_0 = X \in T\Perp{<0}$ as desired.
\end{proof}
The following is a direct consequence of \cref{condition}. We will show in \cref{counterexample} that the converse implication is not true in general.
\begin{cor}
    Let $T$ be a silting complex in $\D(\Mod R)$ such that $T^+$ is a cotilting complex in $\D(\lMod R)$. Then $T$ is tilting.
\end{cor}
\begin{cor}\label{cofinite}
    The assignment $T \mapsto T^+$ of \cref{silting-cosilting} restricts to a bijection
    $$\left \{ \begin{tabular}{ccc} \text{Decent tilting complexes} \\ \text{in $\D(\Mod R)$} \\ \text{up to equivalence} \end{tabular}\right \}  \stackrel{1-1}{\longleftrightarrow}  \left \{ \begin{tabular}{ccc}  \text{Cotilting complexes of cofinite type} \\ \text{in $\D(\lMod R)$} \\ \text{up to equivalence} \end{tabular}\right \}.$$
\end{cor}
\begin{proof}
    Follows directly from \cref{condition}, as a silting complex $T$ is decent (and thus in particular, tilting) if and only if $T^+$ is a cotilting complex. 
\end{proof}
\begin{cor}\label{noeth}
    Let $R$ be a left hereditary or a commutative noetherian ring. The assignment $T \mapsto T^+$ of \cref{silting-cosilting} restricts to a bijection
    $$\left \{ \begin{tabular}{ccc} \text{Decent tilting complexes} \\ \text{in $\D(\Mod R)$} \\ \text{up to equivalence} \end{tabular}\right \}  \stackrel{1-1}{\longleftrightarrow}  \left \{ \begin{tabular}{ccc}  \text{Cotilting complexes} \\ \text{in $\D(\lMod R)$} \\ \text{up to equivalence} \end{tabular}\right \}.$$
\end{cor}
\begin{proof}
    In both the situations, any cosilting complex is automatically of cofinite type by \cite[Theorem 3.11]{AHH19} or \cite[Corollary 2.14]{HN21}, see also \cite[Proposition 3.10]{MV18}.
\end{proof}
\begin{ex}
    Let $M \in \Mod R$. Then $M$ is a silting complex in $\D(\Mod R)$ if and only if $M$ is an $n$-tilting module for some $n \geq 0$ (see e.g. \cite[\S 2]{PS21} for the definition). Then any such $n$-tilting module is well-known to satisfy condition $(ii)$ of \cref{condition}. Together with \cref{def-tilting}, we see that decent tilting complexes generalize both $n$-tilting modules and compact tilting complexes.
\end{ex}
\subsection{Examples over commutative rings} In this subsection, we discuss some sources of non-trivial examples of decent tilting complexes, as well as one non-example, which come from commutative algebra, where classification results for (co)silting complexes are available (\cite{AH17}, \cite{Hrb20}).

A silting complex is \newterm{2-term} if it is isomorphic in $\D(\Mod R)$ to a complex of projectives concentrated in degrees 0 and -1. Following \cite{AH17}, if $R$ is a commutative ring then 2-term silting complexes $T$ in $\D(\Mod R)$ correspond up to equivalence to hereditary torsion pairs $(\Tcal_\Gcal,\Fcal_\Gcal)$ of finite type in $\Mod R$, which in turn correspond to Gabriel filters $\Gcal$ of finite type. Over a commutative ring, a Gabriel filter of finite type is a linear topology of open ideals of $R$ with a base of finitely generated ideals closed under ideal products, see \cite[Lemma 2.3]{Hone}. Then $\Tcal_\Gcal = \{M \in \Mod R \mid \Ann(m) \in \Gcal ~\forall m \in M\}$, where $\Ann(m)$ is the annihilator ideal of $m$. Furthermore, given such torsion pair, the tilting heart $\Hcal_T$ can be described just using cohomology by
$$X \in \Hcal_T \iff H^i(X) \begin{cases} \in \Ecal_\Gcal & i = -1 \\ \in \Dcal_\Gcal & i = 0 \\  = 0 & \text{else} \end{cases},$$ 
where $\Dcal_\Gcal = \{M \in \Mod R \mid M = IM ~\forall I \in \Gcal\}$ is the torsion class of all $\Gcal$-divisible modules and $\Ecal_\Gcal = \Dcal_\Gcal\Perp{0}$ is the corresponding torsion-free class of $\Gcal$-reduced modules, i.e. modules with no non-zero $\Gcal$-divisible submodule. In other words, $\Hcal_T$ is obtained as the Happel-Reiten-Smal\o~ tilt with respect to the torsion pair $(\Dcal_\Gcal,\Ecal_\Gcal)$ in $\Mod R$.

\begin{prop}
    Let $R$ be commutative noetherian ring. Then any 2-term silting complex $T \in \D(\Mod R)$ is decent tilting.
\end{prop}
\begin{proof}
    This follows from \cite[Corollary 5.12]{PV20}, which proves that every 2-term cosilting complex in $\D(\Mod R)$ is cotilting, in combination with \cref{condition}. Here, a cosilting complex is 2-term if it is isomorphic to a complex of injective modules concentrated in degrees 0 and 1, so that $T^+$ is a 2-term cosilting complex.
\end{proof}
\begin{rmk}
    Much more is proved in \cite[Theorem 6.16, see also Corollary 6.17]{PV20}: If $R$ is commutative noetherian, any intermediate and compactly generated {{\tst}}structure in $\D(\Mod R)$ which restricts to a {{\tst}}structure in the bounded derived category $\D^\bdd(\mod R)$ of finitely generated modules is induced by a cotilting complex. Such {{\tst}}structures are abundant (see \cite[\S 5, \S 6]{ATJLS10}), and by \cref{noeth} each of them gives rise to a decent tilting complex.

    Another source of decent tilting complexes over commutative noetherian rings is provided in \cite[\S 6, \S 7]{HNS}. There, it is shown that any codimension function $\mathsf{d}: \Spec R \to \Zbb$ on the Zariski spectrum of a commutative noetherian ring $R$ of finite Krull dimension gives rise to a silting complex $T_\mathsf{d}$ in the derived category. If $R$ is in addition a homomorphic image of a Cohen-Macaulay ring, then $T_\mathsf{d}$ is in fact tilting and its character dual is cotilting. Unless $R$ is itself Cohen-Macaulay, $T_\mathsf{d}$ is not quasi-isomorphic to a stalk complex, \cite[Remark 5.10]{HNS}.
\end{rmk}
Denote by $t_\Gcal: \Mod R \to \Mod R$ the torsion radical corresponding to the hereditary torsion class $\Tcal_\Gcal$. We say that a $\Gcal$-torsion module $M \in \Tcal_\Gcal$ is \newterm{$\Gcal$-bounded} if there is $I \in \Gcal$ such that $IM = 0$.
\begin{lemma}\label{2term}
    Let $R$ be a commutative ring and $T$ be a 2-term silting complex in $\D(\Mod R)$ corresponding to a Gabriel filter $\Gcal$. Then:
    \begin{enumerate}
        \item If $t_\Gcal(R)$ is $\Gcal$-reduced then $T$ is a tilting complex.
        \item $t_\Gcal(R)$ is $\Gcal$-bounded if and only if $T^+$ is a cotilting complex.
    \end{enumerate}
\end{lemma}
\begin{proof}
    We claim that there is an exact sequence
    $$0 \to t_\Gcal(R) \to R \to D_0 \to D_1 \to 0$$
    where $D_0, D_1 \in \Dcal_\Gcal$. As $\Gcal$-divisible modules are closed under epimorphic images, it is enough to see that $R/t_\Gcal(R)$ admits a monomorphism $R/t_\Gcal(R) \to D_0$ to a $\Gcal$-divisible module. This is the case because the induced Gabriel filter $\overline{\Gcal}=\{(I + t_\Gcal(R))/t_\Gcal(R) \mid I \in \Gcal\}$ over the quotient ring $R/t_\Gcal(R)$ is faithful, and so the $\overline{\Gcal}$-divisible pre-envelopes over this quotient ring come from a 1-tilting cotorsion pair, and thus are monomorphic, see \cite[\S 2, Theorem 3.16]{Hone} for details. Then $T$ is tilting by \cite[Theorem A]{CHZ} (cf. \cref{ss-derequivalence}).

    Now we prove the second statement. By \cite[Theorem 5.6]{PV20} in combination with \cite[Corollary 5.2]{PV20}, $T^+$ is cotilting if and only if $J \in \Gcal$, where $J = \trace_{R/t_\Gcal(R)}R$ is the trace ideal of the cyclic module $R/t_\Gcal(R)$. If the $\Gcal$-torsion in $R$ is bounded there is $I \in \Gcal$ such that $I t_\Gcal(R) = 0$. This implies $I \subseteq J$ and so $J \in \Gcal$. On the other hand, clearly $J t_\Gcal(R) = 0$, and so $J \in \Gcal$ implies that the $\Gcal$-torsion of $R$ is bounded.
\end{proof}

We are ready to provide an example of an indecent tilting complex.
\begin{ex}\label{counterexample}
    There is a commutative ring $R$ and a 2-term tilting complex $T$ in $\D(\Mod R)$ such the cosilting complex $T^+$ is not cotilting. Furthermore, we show that in this example the natural map $T^{(\omega)} \to T^\omega$ is not a monomorphism in $\Hcal_T$ (the coproduct and product of copies of $T$ coincides whether computed in $\Hcal_T$ or $\D(\Mod R)$). As a consequence $\Hcal_T$ cannot be equivalent to a category of contramodules over any complete and separated topological ring, see \cite[last paragraph of \S 6.2]{PS21}. 

    The following construction comes from \cite[Example 2.6]{P16}. Let $k$ be a field and $R$ be the commutative $k$-algebra generated by the infinite sequence $x_1,x_2,x_3,\ldots$ and another generator $y$ subject to relations $x_i x_j = 0$ and $y^ix_i = 0$ for all $i,j > 0$. Consider the Gabriel filter $\Gcal$ generated by the principal ideal $(y)$, meaning that it has a filter base formed by all ideals of the form $(y^k)$ for $k > 0$. Note that by the construction, the $\Gcal$-torsion radical $t_\Gcal(R)$ of $R$ is isomorphic to a direct sum $\bigoplus_{n > 0}(x_n)$ of cyclic ideals generated by $x_n$'s, and furthermore, each $(x_n)$ is isomorphic to a module of the form $k[y]/(y^n)$ with the obvious $R$-action in which $x_n$'s act as zero. Then it is straightforward to check that $t_\Gcal(R)$ is not $\Gcal$-bounded, but it is $\Gcal$-reduced, so the induced 2-term silting complex $T$ is tilting, but $T^+$ is not cotilting by \cref{2term}.

    Furthermore, let $f: R \to R[y^{-1}]$ be the localization map. Recall from \cite[Proposition 1.3]{AMVtau} and \cite[Example 4.14(5)]{AHH19} that we can chose $T = R[y^{-1}] \oplus \Cone(f)$ (see also \cite[Proposition 5.15]{AHH19}). We claim that the map $g: T^{(\omega)} \to T^\omega$ is not a monomorphism in $\Hcal_T$, or equivalently, that $\Cone(g) \not \in \Hcal_T$. Clearly, $H^{-1}\Cone(g) \cong \Coker(H^{-1}g)$. Furthermore, $H^{-1}g$ is just the map $t_{\Gcal}(R)^{(\omega)} \to t_{\Gcal}(R)^\omega$. and so $\Coker(H^{-1}g) \cong t_{\Gcal}(R)^\omega/t_{\Gcal}(R)^{(\omega)}$. Since $t_{\Gcal}(R)^\omega/t_{\Gcal}(R)^{(\omega)}$ is precisely the $\omega$-reduced product, a direct limit of any $\omega$-shaped diagram of copies of $t_{\Gcal}(R)$ can be embedded into it \cite[Theorem 3.3.2]{Prest}.

    It remains to show that there is an $\omega$-shaped diagram of copies of $t_{\Gcal}(R)$ whose direct limit is not $\Gcal$-reduced. As $t_\Gcal(R) \cong \bigoplus_{n > 0}(x_n) \cong k[y]/(y^n)$, there is a monic endomorphism $h$ of $t_\Gcal(R)$ which sends $x_n$ to $y x_{n+1}$ for each $n>0$. The direct limit of the system $t_\Gcal(R) \xrightarrow{h} t_\Gcal(R) \xrightarrow{h} t_\Gcal(R) \xrightarrow{h} t_\Gcal(R) \xrightarrow{h} \cdots$ is not $\Gcal$-reduced because multiplication by $y$ acts surjectively on the direct limit (and it is non-zero).
\end{ex}

\section{Good tilting complexes}
Let $T$ be a silting complex in $\D(\Mod R)$ and let $\Sfr = \End_{\D(\Mod R)}(T)$ be the endomorphism ring which we again consider as a topological ring by endowing it with the compact topology. We also let $A = \dgEnd_R(T) \cong \RHom_R(T,T)$ be the endomorphism dg-ring of $T$; this is a weakly non-positive dg-ring as $H^i(A) = \Hom_{\D(\Mod R)}(T,T[i]) = 0$ for $i>0$. We have the identification $\Sfr = H^0(A)$ and the weak non-positivity ensures that there is the standard {{\tst}}structure $(\D^{\leq 0},\D^{>0})$ in the derived category of right dg-modules $\D(\dgMod A)$ defined by vanishing of cohomology: $\D^{\leq 0} = \{X \in \D(\dgMod A) \mid H^i(X) = 0 ~\forall i>0\}$ and $\D^{> 0} = \{X \in \D(\dgMod A) \mid H^i(X) = 0 ~\forall i \leq 0\}$. We denote by $\tau^{\leq 0}$ and $\tau^{>0}$ the associated \newterm{smart truncation} functors, as well as their obvious shifted variants such as $\tau^{\geq 0}$ (see e.g. \cite[7.3.6, 7.3.10]{Yek}). The heart of this {{\tst}}structure is identified with $\Mod \Sfr$ via the zig-zag of dg-ring morphisms $\Sfr \leftarrow \tau^{\leq 0} A \xrightarrow{qis} A$, see e.g. \cite[Remark 1.6]{BM}.

The following definition was introduced for 1-tilting modules by \cite{Bazz10}, and then generalized to $n$-tilting modules by \cite{BMT11}. Our version for silting complexes is akin to the condition used in a more general dg setting by Nicolás and Saorín \cite{NS18}. Recall that a full triangulated subcategory of a triangulated category is \newterm{thick} if it is closed under direct summands, and we denote the thick closure operator as $\thick(-)$.
\begin{dfn}
    A silting complex $T \in \D(\Mod R)$ is \newterm{good} provided that $R \in \thick(T)$.
\end{dfn}
The assumption of being good is satisfied for any compact silting object. For non-compact ones being good is not automatic, but it is still a rather mild assumption thanks to the following well-known observation, which we reprove here in our setting.
\begin{lemma}\label{good}
    For any silting complex $T$ there is a good silting complex $T'$ equivalent to $T$.
\end{lemma}
\begin{proof}
    Since $T$ is a silting complex we have $R \in \thick(\Add(T))$, see \cite[Proposition 5.3]{AH19}. In particular, there are $T_0,\ldots,T_{n-1} \in \Add(T)$ such that $R \in \thick(T_0,\ldots,T_{n-1})$. Put $T' = T \oplus \bigoplus_{i=0}^{n-1}T_i$, then clearly $R \in \thick(T')$ and $\Add(T) = \Add(T')$. The last equality implies that $T'$ is a silting complex and that it is equivalent to $T$ as such.
\end{proof}
We record some adjunction formulas for dg-(bi)modules, well-known to experts, for further use.
\begin{lemma}\label{dgmorphisms}
    Let $A$ and $B$ be dg-rings. 
    
    There is an evaluation morphism 
    $$\gamma_{X,Y,Z}: \RHom_A(X,Y) \otimes_B^\mathbf{L} Z \to \RHom_A(X,Y \otimes_B^\mathbf{L} Z)$$
    which is natural in $X \in \D(\dglMod A), Y \in \D(\lMod{(A \otimes_{\Zbb}B^{\op})}), Z \in \D(\dglMod B)$.

    There is an evaluation morphism 
    $$\delta_{X,Y,Z}: \RHom_A(Y,X) \otimes_B^\mathbf{L} Z \to \RHom_A(\RHom_B(Z,Y),X)$$
    natural in $Y \in \D(\Mod{(A\otimes_{\Zbb}B^{\op})}), X \in \D(\dgMod A), Z \in \D(\dglMod B)$.

    If $Z$ is a compact object in $\D(\dglMod B)$ then both $\gamma_{X,Y,Z}$ and $\delta_{X,Y,Z}$ are isomorphisms (in $\D(\Mod \Zbb)$).
\end{lemma}
\begin{proof}
    For the first morphism we refer to \cite[Theorem 12.9.10,Theorem 14.1.22]{Yek}. The existence of the second morphism is covered e.g. in \cite[Lemma 1.3]{BM}. The fact that $\delta_{X,Y,Z}$ is an isomorphism if $Z$ is compact follows by a standard argument. Indeed, this map is easily checked to be an isomorphism if $Z = B$, and therefore it is an isomorphism also if $Z \in \thick_{\D(\dglMod B)}(B) = \D(\dglMod B)^\cpt$.
\end{proof}
The following results are available in \cite{NS18} (see also \cite{BM}), but we gather the relevant parts here in a form directly applicable for our purposes.
\begin{theorem}\label{fully-faithful}
    Let $T \in \D(\Mod R)$ be a good silting complex. Then:
    \begin{enumerate} 
        \item[(i)] $T$ is compact as an object of $\D(\dglMod A)$,
        \item[(ii)] the canonical morphism $R \to \RHom_A(T,T)$ is a quasi-isomorphism.
        \item[(iii)] both the functors $\RHom_R(T,-):\D(\Mod R) \to \D(\dgMod A)$ and $T \otimes_R^\mathbf{L} -: \D(\lMod R) \to \D(\dglMod A)$ are fully faithful.
        \item[(iv)] the essential images from both the functors from $(iii)$ are thick subcategories. 
    \end{enumerate}
\end{theorem}
\begin{proof}
    $(i):$ There is a finite sequence of triangles witnessing that $R \in \thick_{\D(\Mod R)}(T)$, and by applying $\RHom_R(-,T)$ on these triangles we obtain a sequence of triangles verifying that $T \in \thick_{\D(\dglMod A)}(A) = \D(\dglMod A)^\cpt$.
    
    $(ii):$ Follows as in \cite[Theorem 1.4]{BM}. Indeed, consider the unit morphism $\alpha_X: X \to \RHom_A(\RHom_R(X,T),T)$ of the adjoint pair $\RHom_R(-,T): \D(\Mod R)^{\op} \begin{array}{c} \rightarrow \\[-6pt] \leftarrow \end{array} \D(\dglMod A)^{\op} :\RHom_A(-,T)$. Clearly, 
    $$\alpha_T: T \to \RHom_A(\RHom_R(T,T),T) \cong \RHom_A(A,T) \cong T$$ 
    is an isomorphism. But since $R \in \thick_{\D(\Mod R)} T$, also 
    $$\alpha_R: R \to \RHom_A(\RHom_R(R,T),T) \cong \RHom_A(T,T)$$ 
    is an isomorphism, and this is easily checked to coincide with the canonical morphism.

    $(iii):$ That $\RHom_R(T,-)$ is fully faithful follows from \cite[Theorem 6.4]{NS18} or also \cite[Theorem 1.4]{BM}. The second statement follows similarly and is also contained in \cite[Theorem 6.4]{NS18}, but we sketch it here for convenience. It is enough to show that the unit morphism $X \to \RHom_A(T,T \otimes_R^\mathbf{L}X)$ is an isomorphism for any $X \in \D(\lMod R)$. By \cref{dgmorphisms}, the natural morphism $\gamma_{X,T,T}: \RHom_A(T,T) \otimes_R^\mathbf{L} X \to \RHom_A(T,T \otimes_R^\mathbf{L}X)$ is an isomorphism. Then the above unit morphism factorizes as
    $$X \toeq R \otimes_R^\mathbf{L}X \toeq \RHom_A(T,T) \otimes_R^\mathbf{L} X \xrightarrow{\gamma_{X,T,T}} \RHom_A(T,T \otimes_R^\mathbf{L}X),$$
    and therefore it is an isomorphism.

    $(iv)$: The functors in question are fully faithful by $(iii)$ and they admit a left or right adjoint, respectively, which yields the claim.
\end{proof}
\begin{cor}\label{ff-ctra}
    Let $T \in \D(\Mod R)$ be a good and decent tilting object. Then the forgetful functor $\Ctra \Sfr \to \Mod \Sfr$ is fully faithful.
\end{cor}
\begin{proof}
    The functor $\Hom_{\Hcal_T}(T,-): \Hcal_T \to \Mod \Sfr$ factorizes into the composition of the equivalence $\Hom_{\Hcal_T}(T,-): \Hcal_T \toeq \Ctra \Sfr$ of \cref{heartequiv} and the forgetful functor $\Ctra \Sfr \to \Mod \Sfr$. On the other hand, the functor $\RHom_R(T,-): \D(\Mod R) \to \D(\dgMod A)$ is fully faithful by \cref{fully-faithful}. Then also the restriction of $\RHom_R(T,-)$ to the heart $\Hcal_T$ is fully faithful, and this functor naturally identifies with $\Hom_{\Hcal_T}(T,-): \Hcal_T \to \Mod \Sfr$. Here, we use that $\RHom_R(T,-): \D(\Mod R) \to \D(\dgMod A)$ is {{\tst}}exact (i.e., preserves both left and right constituents of the {{\tst}}structures in consideration) with respect to the tilting {{\tst}}structure in the former category and the standard {{\tst}}structure in the latter, whose heart identifies with $\Mod \Sfr$.
\end{proof}
\subsection{Derived equivalence and realization functors}\label{ss-derequivalence}
Given an abelian category $\Acal$, a {{\tst}}structure $(\Ucal,\Vcal)$ in $\D(\Acal)$ (or in $\D^\bdd(\Acal)$) is \newterm{intermediate} if there are integers $n < m$ such that $\D^{\leq n} \subseteq \Ucal \subseteq \D^{\leq m}$, where again $\D^{\leq n}$ consists of complexes with cohomology vanishing in degrees $>n$. It is easy to check that any intermediate {{\tst}}structure in the unbounded derived category restricts to a {{\tst}}structure in the bounded derived category and that any {{\tst}}structure induced by a silting complex in $\D(\Mod R)$ is intermediate, and similarly for cosilting complexes in $\D(\lMod R)$.

Following \cite{BBD81}, as explained in a more detail in \cite[\S 2, \S 3, \S 4]{PV18}, for any intermediate {{\tst}}structure $\Tbb$ in $\D(\Mod R)$ there is a triangle functor $\real_\Tbb: \D^\bdd(\Hcal_\Tbb) \to \D^\bdd(\Mod R)$ which extends the inclusion $\Hcal_\Tbb \subseteq \D^\bdd(\Mod R)$. Such a \newterm{realization functor} may in principle be non-unique\footnote{This pathology seems to disappear once we switch to a strong enough enhancement of the derived category, cf. \cite[Proposition 1.3.3.7]{HA}.} and in fact it is constructed using, and determined by, a suitable enhancement of $\D^\bdd(\Mod R)$ called the \newterm{f-enhancement}, see \cite[\S 3]{PV18}. An example of an f-enhancement is the structure of a \newterm{filtered (bounded) derived category}, which is always available for $\D^\bdd(\Mod R)$.

Let $T \in \D(\Mod R)$ be a silting complex. Then the abstract tilting theory developed in \cite{PV18} shows that the realization functor $\real_T: \D^\bdd(\Hcal_T) \to \D^\bdd(\Mod R)$ is a triangle equivalence if and only if $T$ is tilting \cite[Corollary 5.2]{PV18}. The analogous result is also true for bounded cosilting objects in $\D^\bdd(\lMod R)$, see \textit{loc. cit}.
\begin{prop}\label{fenh}
    Let $R,S$ be rings and let $\Tbb_R$ and $\Tbb_S$ be two {{\tst}}structures in $\D^\bdd(\Mod R)$ and $\D^\bdd(\Mod S)$ respectively. Let $F: \D^\bdd(\Mod R) \to \D^\bdd(\Mod S)$ be a triangle functor which satisfies the following conditions:
    \begin{enumerate}
        \item[(i)] $F$ is $t$-exact with respect to the {{\tst}}structures $\Tbb_R$ and $\Tbb_S$, that is, $F$ preserves both the left and the right constituents of the {{\tst}}structures,
        \item[(ii)] $F$ is fully faithful and its essential image is a thick subcategory of $\D^\bdd(\Mod S)$.
    \end{enumerate}
    By (i), $F$ restricts to an exact functor $F_0: \Hcal_{\Tbb_R} \to \Hcal_{\Tbb_S}$ between the hearts of the two {{\tst}}structures. Then the following diagram commutes (up to natural equivalence):
    $$
    \begin{tikzcd}
        \D^\bdd(\Hcal_{\Tbb_R}) \arrow{d}{\real_{\Tbb_R}}[swap]{}\arrow{r}{F_0} & \D^\bdd(\Hcal_{\Tbb_S}) \arrow{d}{\real_{\Tbb_S}} \\
        \D^\bdd(\Mod R) \arrow{r}{F} & \D^\bdd(\Mod S)
    \end{tikzcd}
    $$
    where in the upper row $F_0$ is naturally extended to the bounded derived categories and the realization functors are taken with respect to suitable f-enhancements.
\end{prop}
\begin{proof}
    This follows by combining \cite[Theorem 3.13]{PV18} and \cite[Corollary 3.9]{PV18}. Indeed, by \cite[Example 3.2]{PV18}, the bounded derived category $\D^\bdd(\Mod S)$ admits an f-enhancement in the form of a filtered (bounded) derived category, see \cite[\S 3.1]{PV18} for the definitions. By $(ii)$ and \cite[Corollary 3.9]{PV18}, there is an induced f-enhancement on $\D^\bdd(\Mod R)$ such that $F$ admits an f-lifting with respect to the two f-enhancements. Then \cite[Theorem 3.13]{PV18} applies.
\end{proof}
For any morphism $f: A \to B$ of dg-rings such that $f$ is a quasi-isomorphism of the underlying complexes, the forgetful functor $U_f: \D(\dgMod B) \to \D(\dgMod A)$ is a triangle equivalence with $I_f:= - \otimes_A^\mathbf{L} B \cong \RHom_A(B,-): \D(\dgMod A) \to \D(\dgMod B)$, the inverse equivalence. Moreover, both the functors $U_f$ and $I_f$ preserve the cohomology of the objects. For details, see \cite[Theorem 12.7.2, Lemma 12.7.3]{Yek}.

Now let $R$ be a ring, $T \in \D(\Mod R)$ a tilting complex, $\Sfr = \End_{\D(\Mod R)}(T)$ be the endomorphism ring endowed with the compact topology, and let $A = \dgEnd_R(T)$ be the endomorphism dg-ring of $T$. Since $T$ is tilting, we have $H^0(A) = \Sfr$ and $H^i(A) = 0$ for all $i \neq 0$. Taking the smart truncations, there is a zig-zag of quasi-isomorphisms of dg-rings $\Sfr \cong \tau^{\geq 0}\tau^{\leq 0} A \xleftarrow{r} \tau^{\leq 0} A \xrightarrow{l} A$ which induces a triangle equivalence $\epsilon: \D(\dgMod A) \to \D(\Mod \Sfr)$, where $\epsilon = I_r \circ U_l$. Clearly, $\epsilon$ restricts to an equivalence $\D^\bdd(\dgMod A) \to \D^\bdd(\Mod \Sfr)$ of the corresponding bounded derived categories. By abuse of notation, we will denote by $\epsilon$ also the analogous equivalence $\D(\dglMod A) \toeq \D(\lMod \Sfr)$ on the side of left (dg-)modules.

Note that $T$ being isomorphic in $\D(\Mod R)$ to a bounded complex of projectives implies that both the functors $\RHom_R(T,-)$ and $T \otimes_R^\mathbf{L}-$ restrict to functors between the respective bounded derived categories.
\begin{theorem}\label{tilting-eq}
    Assume that $T$ is good and decent. Then the forgetful functor $\D^\bdd(\Ctra \Sfr) \to \D^\bdd(\Mod \Sfr)$ is fully faithful, and the functor $\phi = \epsilon \circ \RHom_R(T,-)$ induces a triangle equivalence $\phi: \D^\bdd(\Mod R) \to \D^\bdd(\Ctra \Sfr)$ with inverse equivalence $(- \otimes_A^\mathbf{L} T)\circ (\epsilon^{-1})_{\restriction \D(\Ctra \Sfr)}$.
\end{theorem}
\begin{proof}
    Observe that the functor $\phi$ restricted to $\Hcal_T$ is equivalent to the functor $\Hom_{\Hcal_T}(T,-): \Hcal_T \to \Ctra \Sfr \subseteq \Mod \Sfr$ of \cref{heartequiv}, the last inclusion being fully faithful is provided by \cref{ff-ctra}. Since $T$ is a tilting complex, any realization functor $\real_T: \D^\bdd(\Hcal_T) \to \D^\bdd(\Mod R)$ is a triangle equivalence \cite[Corollary 5.2]{PV18}. The functor $\phi$ is clearly {{\tst}}exact with respect to the tilting {{\tst}}structure $(T\Perp{>0},T\Perp{\leq 0})$ in $\D(\Mod R)$ and the standard {{\tst}}structure in $\D(\Mod \Sfr)$. By \cref{fully-faithful}, $\phi$ is a fully faithful functor realizing $\D^\bdd(\Mod R)$ as a thick subcategory of $\D^\bdd(\Mod \Sfr)$. Then \cref{fenh} applies and yields a commutative diagram as follows:
    $$
    \begin{tikzcd}
        \D^\bdd(\Hcal_T) \arrow{d}{\real_T}[swap]{\cong}\arrow{r}{\Hom_{\Hcal_T}(T,-)} & \D^\bdd(\Mod \Sfr) \arrow{d}{=} \\
        \D^\bdd(\Mod R) \arrow{r}{\phi} & \D^\bdd(\Mod \Sfr) 
    \end{tikzcd}
    $$
    Since $\phi$ is fully faithful, and since $\Hom_{\Hcal_T}(T,-): \D^\bdd(\Hcal_T) \to \D^\bdd(\Mod \Sfr)$ factorizes through the equivalence $\D^\bdd(\Hcal_T) \toeq \D^\bdd(\Ctra \Sfr)$ obtained from \cref{heartequiv}, it follows that the forgetful functor $\D^\bdd(\Ctra \Sfr)\to \D^\bdd(\Mod \Sfr)$ is fully faithful. Then the essential image of $\Hom_{\Hcal_T}(T,-): \D^\bdd(\Hcal_T) \to \D^\bdd(\Mod \Sfr)$ is precisely the full subcategory $\D^\bdd(\Ctra \Sfr)$ and the commutative square above restricts to another commutative square 
    $$
    \begin{tikzcd}
        \D^\bdd(\Hcal_T) \arrow{d}{\real_T}[swap]{\cong}\arrow{r}{\Hom_{\Hcal_T}(T,-)}[swap]{\cong} & \D^\bdd(\Ctra \Sfr) \arrow{d}{=} \\
        \D^\bdd(\Mod R) \arrow{r}{\phi} & \D^\bdd(\Ctra \Sfr) 
    \end{tikzcd}
    $$
    It follows that $\phi$, viewed as a functor $\D^\bdd(\Mod R) \to \D^\bdd(\Ctra \Sfr)$, is a triangle equivalence.

    The inverse equivalence is $(- \otimes_A^\mathbf{L} T)\circ (\epsilon^{-1})_{\restriction \D^\bdd(\Ctra \Sfr)}$ since this is the left adjoint to $\phi$ and $\D^\bdd(\Ctra \Sfr)$ is a full subcategory of $\D^\bdd(\Mod \Sfr)$.
\end{proof}
\begin{rmk}\label{good-necessary}
    Similarly to \cite[Proposition 8.2]{PS21}, we need the assumption of $T$ being good in order to represent the derived equivalence $\D^\bdd(\Mod R) \cong \D^\bdd(\Ctra \Sfr)$ in terms of the derived functors $\RHom_R(T,-)$ and $- \otimes_A^\mathbf{L} T$. To achieve this, such an assumption is necessary, see \cite[Theorem 6.4]{NS18}. However, the mere existence of a triangle equivalence $\D^\bdd(\Mod R) \cong \D^\bdd(\Ctra \Sfr)$ does not depend on $T$ being good and follows directly from \cite[Corollary 5.2]{PV20} and the equivalence $\Hcal_T \cong \Ctra \Sfr$ of \cref{heartequiv}, or directly from \cref{gen-der-eq}. In fact, the work of Virili \cite[Theorem 7.12]{SV18} shows that we even have the unbounded derived equivalence $\D(\Mod R) \cong \D(\Ctra \Sfr)$. We do not address the issue of representing the unbounded derived equivalence by derived functors in this paper, but we remark that this is available for the tilting module case in \cite[Proposition 8.2]{PS21}.

    Here, we should also explain that passing to a good representative in the equivalence class of a decent tilting complex leads not only to an equivalent tilting heart, but also to equivalent category of contramodules. Indeed, for any non-empty set $X$, the two rings $\End_{\D(\Mod R)}(T)$ and $\End_{\D(\Mod R)}(T^{(X)})$ equipped with compact topologies have equivalent contramodule categories, see \cite[Remark 5.5]{PS19b} and the discussion in \cref{ss-Morita} for more details.
\end{rmk}

In what follows, an easy corollary of the previous result will be useful, which is obtained simply by suitably restricting the equivalence of \cref{tilting-eq} to the corresponding hearts.
\begin{cor}\label{ab-eq}
    In the setting as above, the equivalence of \cref{tilting-eq} restricts to an equivalence $H^0\RHom_{R}(T,-): \Hcal_T \cong \Ctra \Sfr : - \otimes_A^\mathbf{L} T$, when we consider $\Ctra \Sfr$ as a full subcategory of the heart of the standard {{\tst}}structure in $\D(\dgMod A)$.
\end{cor}
\section{Cotilting hearts and discrete modules}
In this section, we consider the cotilting complex $C = T^+$ associated to a decent tilting complex $T$ via \cref{condition} and show that the cotilting heart $\Hcal_C$ is equivalent to another category induced by the compact topology induced by $T$, the category of left discrete modules.
\subsection{Discrete modules and contratensor product}
Let $\Rfr$ be a (left) topological ring. A left $\Rfr$-module $N$ is called \newterm{discrete} if for any element $n \in N$ its annihilator $\Ann_\Rfr(n) = \{r \in \Rfr \mid rn = 0\}$ is open. The full subcategory $\Disc \Rfr$ of $\lMod \Rfr$ consisting of all discrete left $\Rfr$-modules is a locally finitely generated Grothendieck category with an injective cogenerator obtained by taking the maximal discrete submodule of an injective cogenerator of $\lMod \Rfr$.

Now assume that $\Rfr$ is complete and separated. For any $N \in \Disc \Rfr$ and any abelian group $V$, the right $\Rfr$-module $\Hom_{\Mod \Zbb}(N,V)$ structure naturally extends to a structure of a right $\Rfr$-contramodule. There is the \newterm{contratensor product} functor $- \ctrtensor_\Rfr -: \Ctra \Rfr \times \Disc \Rfr \to \Mod \Zbb$, which is defined as a suitable quotient of the ordinary tensor product. The defining property of this functor is the Contratensor-Hom adjunction isomorphism $\Hom_{\Mod \Zbb}(\Mfr \ctrtensor_\Rfr N,V) \cong \Hom_{\Ctra \Rfr}(\Mfr,\Hom_{\Mod \Zbb}(N,V))$. We refer to \cite[\S 7.2]{PS21} for details. Finally, we record the following observation.
\begin{lemma}\label{ff-ctratensor}
Let $T$ be a good and decent tilting complex in $\D(\Mod R)$ and $\Sfr$ as before. Then the contratensor product $- \ctrtensor_\Sfr -: \Ctra \Sfr \times \Disc \Sfr \to \Mod \Zbb$ is naturally equivalent to (the restriction of) the ordinary tensor product $- \otimes_\Sfr -$.
\end{lemma}
\begin{proof}
    By \cref{ff-ctra}, the forgetful functor $\Ctra \Rfr \to \Mod \Rfr$ is fully faithful and this implies our claim, see \cite[Lemma 7.11]{PS21}.
\end{proof}
\subsection{Cotilting hearts}
We are ready to describe cotilting hearts obtained from character duals of decent tilting complexes as the corresponding categories of discrete modules.
\begin{theorem}\label{cotilting-heart}
    Let $T \in \D(\Mod R)$ be a good and decent tilting complex and $C = T^+$. The functor $H^0(T \otimes_R^\mathbf{L} -)$ induces an equivalence $\Hcal_C \to \Disc \Sfr$ with the inverse equivalence $\RHom_A(T,-)$, where we consider $\Disc \Sfr$ as a full subcategory of the heart $\lMod \Sfr$ of the standard {{\tst}}structure in $\D(\dglMod A)$.
\end{theorem}
\begin{proof}
We start by checking that the two functors are well-defined. First, the functor $H^0(T \otimes_R^\mathbf{L} -)$ constitutes a well-defined functor $\D(\lMod R) \to \lMod \Sfr$. This is because the functor $T \otimes_R^\mathbf{L} -$ takes values in $\D(\dglMod A)$, which is sent to $\lMod \Sfr$ by $H^0$. We show that this functor lands in the full subcategory $\Disc \Sfr$; in fact, we show that $H^0(T \otimes_R^\mathbf{L} X)$ is a discrete left $\Sfr$-module for any $X \in \D(\lMod R)$. First, assume that $X$ is a compact object of $\D(\lMod R)$. The compactness of $X$ yields an isomorphism $H^0(T \otimes_R^\mathbf{L} X) \cong H^0 \RHom_R(X^*,T)$ where $X^* = \RHom_R(X,R)$ is a compact object in $\D(\Mod R)$, see \S 3 or \cref{dgmorphisms}. But then $H^0 \RHom_R(X^*,T) = \Hom_{\D(\Mod R)}(X^*,T)$ is clearly a discrete left $\Sfr$-module by the definition of the compact topology on $\Sfr$. Now consider a general object $X \in \D(\lMod R)$, we may and will assume that $X$ is a dg-flat complex. By \cite[Theorem]{CH15}, we can write $X$ as a direct limit $X = \varinjlim_{i \in I}X_i$ of complexes $X_i$ which are compact (in fact, perfect). Then $H^0(T \otimes_R^\mathbf{L} X) \cong H^0(T \otimes_R X) = H^0(T \otimes_R \varinjlim_{i \in I}X_i) \cong \varinjlim_{i \in I} H^0(T \otimes_R^\mathbf{L} X_i)$. Since a direct limit of discrete modules is discrete, this argument is finished.

On the other hand, let $N \in \Disc \Sfr$, then by compactness of $T \in \D(\dglMod A)$ we have that $\RHom_A(T,N)^+ \cong N^+ \otimes_A^\mathbf{L} T$ (see \cref{dgmorphisms}), and so $\RHom_A(T,N)^+ \in \Hcal_T \subseteq \D(\Mod R)$ by \cref{ab-eq} since $N^+ \in \Ctra \Sfr$. In view of \cref{heartduality}, $\RHom_A(T,-)$ induces a well-defined functor $\Disc \Sfr \to \Hcal_C$. Moreover, since $(T \otimes_R^\mathbf{L} X)^+ \cong \RHom_R(X,C)$ and $\Hcal_C = \Perp{\neq 0}C$, we have that $T \otimes_R^\mathbf{L} X$ has cohomology concentrated in degree 0 for any $X \in \Hcal_C$. Therefore, $H^0(T \otimes^\mathbf{L}_R -)$ is naturally identified with $T \otimes_R^\mathbf{L} -$ as functors from the heart $\Hcal_C$, and so the two functors from the statement are well-defined and also mutually adjoint. Therefore, to establish the equivalence it is enough to show that both the unit and the counit morphism of this adjunction are isomorphisms.

Let $N \in \Disc \Sfr$, and consider the counit morphism $\nu: T \otimes_R^\mathbf{L} \RHom_A(T,N) \to N$. To show that $\nu$ is an isomorphism, it suffices to show that $\nu^+: N^+ \to (T \otimes_R^\mathbf{L} \RHom_A(T,N))^+$ is an isomorphism. We can reinterpret $\nu^+$ using the following commutative diagram:
$$
\begin{tikzcd}
    N^+ \arrow[equal]{d}\arrow{r}{\nu^+} & (T \otimes_R^\mathbf{L} \RHom_A(T,N))^+ \\
    N^+ \arrow[equal]{d}\arrow{r} & \RHom_R(T,\RHom_A(T,N)^+)\arrow{u}{\cong}[swap]{\mathsf{adj}} \\
    N^+ \arrow{r}{}                 & \RHom_R(T,N^+ \otimes_A^\mathbf{L} T)\arrow{u}{\cong}[swap]{\RHom_R(T,\delta_{\Qbb/\Zbb,N,T})}
\end{tikzcd}
$$
The isomorphism $\mathsf{adj}$ comes from the suitable Tensor-Hom adjunction, while $\delta_{\Qbb/\Zbb,N,T}$ is the evaluation morphism of \cref{dgmorphisms}, which is an isomorphism because $T$ is compact in $\D(\dgMod A)$. It is straightforward to check that the induced bottom horizontal morphism is identified with the unit morphism of the Tensor-Hom adjunction induced by the dg-bimodule $T$ evaluated at $N^+ \in \Ctra \Sfr \subseteq \D(\dgMod A)$, and this is an isomorphism by \cref{ab-eq}. Recall here that the restriction of the functor $\RHom_R(T,-): \D(\Mod R) \to \D(\dgMod A)$ to $\Hcal_T$ is identified with the functor $H^0\RHom_R(T,-): \Hcal_T \cong \Ctra \Sfr$ of \cref{ab-eq}.

Let $X \in \Hcal_C$ and consider the unit morphism $\eta: X \to \RHom_A(T,T \otimes_R^\mathbf{L} X)$. Similarly as in the previous paragraph, we have the following natural isomorphisms induced by the Tensor-Hom adjunction and \cref{dgmorphisms}:
$$\RHom_A(T,T \otimes_R^\mathbf{L} X)^+ \cong (T \otimes_R^\mathbf{L} X)^+ \otimes_A^\mathbf{L} T \cong$$
$$\cong \RHom_R(T,X^+) \otimes_A^\mathbf{L} T.$$
Arguing as above, it follows that $\eta^+$ identifies with the counit morphism of the Tensor-Hom equivalence of \cref{ab-eq} evaluated at $X^+ \in \Hcal_T$, and therefore $\eta$ is an isomorphism.
\end{proof}
As a consequence, we obtain a description of the hearts induced by cotilting complexes of cofinite type. Recall from \cref{noeth} that if $R$ is left hereditary or commutative noetherian then any cotilting complex is of cofinite type.
\begin{cor}\label{cotilting-cofinite}
    Let $C \in \D(\lMod R)$ be a cotilting complex of cofinite type. Then there is a complete and separated topological ring $\Sfr$ such that $\Hcal_C \cong \Disc \Sfr$.
\end{cor}
\begin{proof}
    By \cref{cofinite}, there is a decent tilting complex $T \in \D(\Mod R)$ such that $C$ is equivalent to $T^+$ as cosilting complexes. By \cref{good}, there is a good and decent tilting complex $T'$ which is equivalent to $T$ as silting complexes. Putting $\Sfr = \End_{\D(\Mod R)}(T')$ and endowing $\Sfr$ with the compact topology, the proof is finished by noting that $\Hcal_C = \Hcal_{T'^+} \cong \Disc \Sfr$ where the last equivalence follows from \cref{cotilting-heart}.
\end{proof}
\subsection{Cotilting derived equivalence}
The following is the cotilting counterpart of \cref{tilting-eq}.
\begin{theorem}\label{cotilting-eq}
    Assume that $T \in \D(\Mod R)$ is a good and decent tilting complex and let again $\Sfr = \End_{\D(\Mod R)}(T)$. Then the forgetful functor $\D^\bdd(\Disc \Sfr) \to \D^\bdd(\lMod \Sfr)$ is fully faithful and the functor $\psi = \epsilon \circ (T \otimes_R^\mathbf{L} -)$ induces a triangle equivalence $\psi: \D^\bdd(\lMod R) \to \D^\bdd(\Disc \Sfr)$.
\end{theorem}
\begin{proof}
    This is proved similarly to \cref{tilting-eq}. Put $C=T^+$ and recall from \cref{condition} that this is a cotilting complex in $\D(\lMod R)$. By \cref{fully-faithful}, $(T \otimes_R^\mathbf{L} -): \D^\bdd(\lMod R) \to \D^\bdd(\dglMod A)$ is fully faithful, and thus so is $\psi$. Furthermore, $\psi$ is clearly $t$-exact with respect to the cotilting t-structure $(\Perp{\leq 0}C, \Perp{>0}C)$ in $\D(\lMod R)$ and the standard t-structure in $\D(\lMod \Sfr)$; recall that $\RHom_R(-,C) = \RHom_R(-,T^+)$ is identified with $(T \otimes_R^\mathbf{L} - )^+$ by adjunction. It follows that $\psi$ restricts to an exact functor $\tau: \Hcal_C \to \lMod \Sfr$, which identifies with the restriction of $H^0(T \otimes_R^\mathbf{L}-)$ to $\Hcal_C$. Note that $H^0(T \otimes_R^\mathbf{L}-)$ further identifies with $H^0(\Psi)$ as a functor $\D^\bdd(\lMod R) \to \lMod \Sfr$ as the triangle equivalence $\epsilon: \D(\dglMod A) \toeq \D(\lMod \Sfr)$ is t-exact with respect to the two standard t-structures and restricts to the identity functor on the heart $\lMod \Sfr$.
    
    Arguing as in \cref{tilting-eq} using \cref{cotilting-heart} this time, \cref{fenh} yields a commutative square:
    $$
    \begin{tikzcd}
        \D^\bdd(\Hcal_C) \arrow{d}{\real_C}[swap]{\cong}\arrow{r}{\tau} & \D^\bdd(\lMod \Sfr) \arrow{d}{=} \\
        \D^\bdd(\lMod R) \arrow{r}{\psi} & \D^\bdd(\lMod \Sfr) 
    \end{tikzcd}
    $$
    Since $\tau$ factorizes as $\Hcal_C \toeq \Disc \Sfr \subseteq \lMod \Sfr$, we obtain that the forgetful functor $\D^\bdd(\Disc \Sfr) \to \D^\bdd(\lMod \Sfr)$ is fully faithful. Then the square above induces another commutative square
    $$
    \begin{tikzcd}
        \D^\bdd(\Hcal_C) \arrow{d}{\real_C}[swap]{\cong}\arrow{r}{\tau} & \D^\bdd(\Disc \Sfr) \arrow{d}{=} \\
        \D^\bdd(\lMod R) \arrow{r}{\psi} & \D^\bdd(\Disc \Sfr) 
    \end{tikzcd}
    $$
    where the upper arrow is an equivalence, making $\psi: \D^\bdd(\lMod R) \to \D^\bdd(\Disc \Sfr)$ an equivalence.
\end{proof}
\begin{rmk}
    As in the tilting case of \cref{good-necessary}, we remark that the assumption of $T$ being good is needed to represent the derived equivalence $\D^\bdd(\lMod R) \to \D^\bdd(\Disc \Sfr)$ by $T \otimes_R^\mathbf{L} -$, but the mere existence of such a triangle equivalence is guaranteed without such assumption. This follows from \cite[Corollary 5.2]{PV20} together with the equivalence $\Hcal_{T^+} \cong \Disc \Sfr$ of \cref{cotilting-heart}. Again, \cite[Theorem 7.12]{SV18} also yields the unbounded derived equivalence $\D(\lMod R) \to \D(\Disc \Sfr)$.
\end{rmk}
\section{Tensor compatibility and the main result}
In Rickard's derived Morita theory for rings, compact tilting objects capture the derived equivalences between module categories of rings. The main aim of this section is to provide a topological algebra generalization by showing that the pairs of derived equivalences between a ring and a topological ring of \cref{tilting-eq} and \cref{cotilting-eq} which are tied together by a suitable tensor compatibility condition are parametrized by decent tilting complexes.
\subsection{Tensor compatiblity} The following observation can be seen as a large tilting version of the Rickard's result on tensor compatibility of representable derived equivalences in the classical setting of \cite{Ri91} (see \cref{I1} from the Introduction).
\begin{theorem}\label{thm-tensor}
    In the setting of \cref{tilting-eq} and \cref{cotilting-eq}, there is a commutative square as follows:
    $$
    \begin{tikzcd}
        \D^\bdd(\Mod R) \times \D^\bdd(\lMod R) \arrow{d}{\phi \times \psi}[swap]{\cong} \arrow{r}{- \otimes_R^\mathbf{L} -} & \D(\Mod \Zbb) \arrow{d}{=} \\
        \D^\bdd(\Ctra \Sfr) \times \D^\bdd(\Disc \Sfr) \arrow{r}{- \ctrtensor_{\Sfr}^\mathbf{L} -} & \D(\Mod \Zbb) 
    \end{tikzcd}
    $$
\end{theorem}
\begin{proof}
    Note first that here the contratensor product $-\ctrtensor_{\Sfr}-$ identifies with the restriction of the ordinary tensor product $-\otimes_{\Sfr}-$ by \cref{ff-ctratensor}. For $X \in \D^\bdd(\Mod R)$ and $Y \in \D^\bdd(\lMod R)$, we have a sequence of natural isomorphisms in $\D(\Mod \Zbb)$ as follows:
    $$\phi(X) \otimes_\Sfr^\mathbf{L} \psi(Y) = \epsilon(\RHom_R(T,X)) \otimes^\mathbf{L}_\Sfr \epsilon(T \otimes_R^\mathbf{L} Y) \cong \RHom_R(T,X) \otimes^\mathbf{L}_A (T \otimes_R^\mathbf{L} Y) \cong$$
    $$\cong (\RHom_R(T,X) \otimes_A^\mathbf{L} T) \otimes_R^\mathbf{L} Y \cong X \otimes_R^\mathbf{L} Y.$$
    In the second isomorphism we use the fact that the equivalence $\epsilon: \D(\dgMod A) \to \D(\Mod \Sfr)$ (resp. $\epsilon: \D(\dglMod A) \to \D(\lMod \Sfr)$) preserves quasi-isomorphism and derived tensor products, see \cite[Theorem 12.7.2]{Yek}. The last isomorphism follows from \cref{fully-faithful}(iii).
\end{proof}
\subsection{Topological Morita theory}\label{ss-Morita}
We need to recall the (non-derived) topological Morita theory\footnote{Note that the topological Morita theory is more general than the ordinary Morita theory even if we start with a discrete ring, because here we are allowed to take projective generators which are not finitely generated.} of complete and separated topological ring developed in \cite{PS21} and \cite{PS19b}. Let $\Rfr$ be a complete and separated (left) topological ring and $\Pfr \in \Ctra \Rfr$ a projective generator. Then the endomorphism ring $\Rfr' = \End_{\Ctra \Rfr}(\Pfr)$ admits a naturally induced linear topology of open left ideals such that $\Rfr'$ is complete and separated and there is an equivalence $\Hom_{\Ctra \Rfr}(\Pfr,-): \Ctra \Rfr \cong \Ctra{\Rfr'}$ which takes $\Pfr$ to $\Rfr'$, \cite[Theorem 7.9, Corollary 6.3]{PS21}. In this situation, we say that $\Rfr$ and $\Rfr'$ are \newterm{topologically Morita equivalent}. In particular, if $T \in \D(\Mod R)$ is a tilting complex with $\Rfr = \End_{\D(\Mod R)}(T)$ endowed with the compact topology and $T'$ is a tilting complex equivalent to $T$ then $\Rfr' = \End_{\D(\Mod R)}(T')$ admits a linear topology which makes it topologically Morita equivalent to $\Rfr$. Furthermore, one can check directly that this topology described in \cite[Corollary 7.7]{PS21} coincides with the compact topology defined on the endomorphism ring $\Rfr'$, see also \cite[Remark 5.5]{PS19b}. We will mainly be interested in the case $\Pfr = \Rfr^{(X)}$ for a non-empty set $X$, and then there is also an equivalence $\Pfr \ctrtensor_\Rfr - : \Disc \Rfr \cong \Disc {\Rfr'}$ of the discrete module categories available from \cite[Proposition 5.2]{PS19b}. Finally, these two equivalences are jointly compatible with the contratensor structure: Let $\Mfr \in \Ctra \Rfr$, $N \in \Disc \Rfr$, and $V \in \Mod \Zbb$, then we have the natural isomorphism $\Hom_{\Ctra \Rfr}(\Mfr,\Hom_{\Zbb}(N,V)) \cong \Hom_{\Ctra \Rfr'}(\Hom_{\Ctra \Rfr}(\Pfr,\Mfr),\Hom_{\Ctra \Rfr}(\Pfr,\Hom_{\Zbb}(N,V)))$ induced by the equivalence. Applying the Contratensor-Hom adjunction three times, we obtain another natural isomorphism $\Hom_{\Zbb}(\Mfr \ctrtensor_\Rfr N,V) \cong \Hom_{\Zbb}(\Hom_{\Ctra \Rfr}(\Pfr,\Mfr) \ctrtensor_{\Rfr'} (\Pfr \ctrtensor_\Rfr N),V)$, and thus we obtain a natural isomorphism $\Mfr \ctrtensor_\Rfr N \cong \Hom_{\Ctra \Rfr}(\Pfr,\Mfr) \ctrtensor_{\Rfr'} (\Pfr \ctrtensor_\Rfr N)$ by Yoneda.
\subsection{Main result}
If $\Acal$ is an abelian category with enough projectives, let $\K^\bdd(\Acal_\Proj)$ denote the homotopy category of bounded complexes of projectives objects of $\Acal$, considered as a full subcategory of $\D^\bdd(\Acal)$. 
\begin{theorem}\label{converse}
    Let $R$ be a ring and $\Sfr$ a complete, separated topological ring. The following are equivalent:
    \begin{enumerate}
        \item[(i)] There is a pair of triangle equivalences $\alpha:\D^\bdd(\Mod R) \toeq \D^\bdd(\Ctra \Sfr)$ and $\beta:\D^\bdd(\lMod R) \toeq \D^\bdd(\Disc \Sfr)$ which make the following diagram commute:
        $$
        \begin{tikzcd}
            \D^\bdd(\Mod R) \times \D^\bdd(\lMod R) \arrow{d}{\alpha \times \beta}[swap]{\cong} \arrow{r}{- \otimes_R^\mathbf{L} -} & \D(\Mod \Zbb) \arrow{d}{=} \\
            \D^\bdd(\Ctra \Sfr) \times \D^\bdd(\Disc \Sfr) \arrow{r}{- \ctrtensor_{\Sfr}^\mathbf{L} -} & \D(\Mod \Zbb) 
        \end{tikzcd}
        $$
        \item[(ii)] There is a decent tilting complex $T \in \D(\Mod R)$ such that $\End_{\D(\Mod R)}(T)$ endowed with the compact topology is isomorphic as a linear topological ring to $\Sfr$.
    \end{enumerate}
\end{theorem}
\begin{proof}
    $(ii) \implies (i):$ Let $X$ be a non-empty set such that $T' = T^{(X)}$ is a good silting complex, which exists by \cref{good}, and let $\Sfr' = \End_{\D(\Mod R)}(T')$ be endowed with the compact topology. \cref{tilting-eq} and \cref{cotilting-eq} yield a pair of derived equivalences $\phi$ and $\psi$ which fit into the required commutative square for $\Sfr'$ by \cref{thm-tensor}. By the discussion above, we can put $\alpha= \mu^{-1} \circ \phi$ and $\beta = \nu^{-1} \circ \psi$, where $\mu = \Hom_{\Ctra \Sfr}(\Sfr^{(X)},-): \Ctra \Sfr \cong \Ctra {\Sfr'}$ and $\nu = \Sfr^{(X)} \ctrtensor_\Sfr - : \Disc \Sfr \cong \Disc {\Sfr'}$ are the topological Morita theory equivalences associated to the projective generator $\Sfr^{(X)}$ of $\Ctra \Sfr$, trivially extended to the derived category level.

    $(i) \implies (ii):$ Recall that $\Sfr$ is a projective generator of $\Ctra \Sfr$ and let $W$ be any injective cogenerator of $\Disc \Sfr$. There is the standard {{\tst}}structure $(\D^{\leq 0},\D^{>0})$ in $\D^\bdd(\Ctra \Sfr)$ which coincides with $(\Sfr\Perp{>0},\Sfr\Perp{\leq 0})$, similarly the (shifted) standard {{\tst}}structure $(\D^{< 0},\D^{\geq 0})$ in $\D^\bdd(\Disc \Sfr)$ coincides with $(\Perp{\leq 0}W,\Perp{>0}W)$. The projectivity of $\Sfr$ (resp. injectivity of $W$) implies that taking coproduct of copies of $\Sfr$ (resp. product of copies of $W$) is the same whether computed on the derived or non-derived level in the respective categories. 
    
    Denote $T = \alpha^{-1}(\Sfr) \in \D^\bdd(\Mod R)$ and $C = \beta^{-1}(W) \in \D^\bdd(\lMod R)$. The equivalences transfer the standard {{\tst}}structures to {{\tst}}structures of the form $\Tbb_T=(T\Perp{>0},T\Perp{\leq 0})$ and $\Tbb_C=(\Perp{\leq 0}C,\Perp{>0}C)$ in $\D^\bdd(\Mod R)$ and $\D^\bdd(\lMod R)$, respectively. 
    The last sentence of the previous paragraph also ensures that $\Add(T) \subseteq \Hcal_T = T\Perp{\neq 0}$ and $\Prod(C) \subseteq \Hcal_C = \Perp{\neq 0}C$, and $\Hcal_T$ and $\Hcal_C$ are the hearts of the two {{\tst}}structures. In other words, $T$ is a tilting object in the triangulated category $\D^\bdd(\Mod R)$ and $C$ is a cotilting object in $\D^\bdd(\lMod R)$. Next, we need to argue that $T$ is a tilting complex in $\D(\Mod R)$ and $C$ a cotilting complex in $\D(\lMod R)$.
    
    The equivalence $\alpha: \D^\bdd(\Mod R) \toeq \D^\bdd(\Ctra \Sfr)$ restricts to $\K^\bdd(\Modpr R) \toeq \K^\bdd(\Ctrapr{ \Sfr})$. Indeed, as in the proof of \cite[Theorem 5.3]{PV18} which follows the argument of \cite[Proposition 6.2]{Ri89}, if $\Acal$ is a cocomplete abelian category with enough projectives then we can characterize $\K^\bdd(\Acal_\Proj)$ inside $\D^\bdd(\Acal)$ internally as a full subcategory consisting of those objects $X$ such that for any $Y$ we have $\Hom_{\D^\bdd(\Acal)}(X,Y[i])$ for $i \gg 0 $. It follows in the same fashion as in \cite[Theorem 5.3]{PV18} that $T$ is isomorphic in $\D^\bdd(\Mod R)$ to a bounded complex of projective $R$-modules and furthermore, that $\thick(\Add(T)) = \K^\bdd(\Modpr R)$. Then by \cite[Proposition 5.3]{AHH19}, together with what we found in the preceding paragraph, $T$ is a tilting complex in $\D(\Mod R)$. An analogous argument using \cite[Proposition 6.8]{AHH19} shows that $C$ is a cotilting complex in $\D(\lMod R)$.

    By an adjunction argument like in the proof of \cref{heartduality}, we have $\Hom_{\D(\lMod R)}(X,T^+[i]) = 0$ if and only if $H^{-i}(T \otimes_R^\mathbf{L} X) = 0$. Via the assumed identification of $T \otimes_R^\mathbf{L} -$ and $\Sfr \ctrtensor_\Sfr^\mathbf{L} \beta(-)$, we see that the equivalence $\beta: \D^\bdd(\lMod R) \toeq \D^\bdd(\Disc \Sfr)$ identifies the (shifted) standard {{\tst}}structure in $\D^\bdd(\Disc \Sfr)$ with the cosilting {{\tst}}structure induced by $T^+$. It follows that the cosilting complexes $C$ and $T^+$ are equivalent, and so in particular, $T^+$ is cotilting and $\Hcal_C = \Hcal_{T^+}$. Then \cref{condition} implies that $T$ is decent. 
    
    We have $\Sfr = \End_{\D(\Mod R)}(T)$ as ordinary rings. Let $\Rfr = \End_{\D(\Mod R)}(T)$ be the same ring as $\Sfr$ but endowed with the compact topology. By the implication $(ii) \implies (i)$, there are derived equivalences $\gamma: \D^\bdd(\Mod R) \cong \D^\bdd(\Ctra \Rfr)$ and $\delta: \D^\bdd(\lMod R) \cong \D^\bdd(\Disc \Rfr)$ fitting into a commutative square as in $(i)$. Combining this commutative square with the one from the assumption $(i)$, we obtain the following commutative square:
    $$
    \begin{tikzcd}
        \D^\bdd(\Ctra {\Rfr}) \times \D^\bdd(\Disc {\Rfr}) \arrow{d}{\alpha\gamma^{-1} \times \beta\delta^{-1}}[swap]{\cong} \arrow{r}{- \ctrtensor^\mathbf{L}_{\Rfr} -} & \D(\Mod \Zbb) \arrow{d}{=} \\
        \D^\bdd(\Ctra \Sfr) \times \D^\bdd(\Disc \Sfr) \arrow{r}{- \ctrtensor^\mathbf{L}_{\Sfr} -} & \D(\Mod \Zbb)
    \end{tikzcd}
    $$
    Using the {{\tst}}exactness of the involved equivalences, this square further restricts to another commutative square:
    $$
    \begin{tikzcd}
        \Ctrapr {\Rfr} \times \Disc {\Rfr} \arrow{d}{\alpha\gamma^{-1} \times \beta\delta^{-1}}[swap]{\cong} \arrow{r}{- \ctrtensor_{\Rfr} -} & \Mod \Zbb \arrow{d}{=} \\
        \Ctrapr \Sfr \times \Disc \Sfr \arrow{r}{- \ctrtensor_{\Sfr} -} & \Mod \Zbb 
    \end{tikzcd}
    $$
    It follows that for any $N \in \Disc \Rfr$, we have natural isomorphisms in $\Mod \Zbb$:
    $$\Rfr \ctrtensor_\Rfr N \cong \alpha\gamma^{-1}(\Sfr) \otimes_{\Sfr} \beta\delta^{-1}(N) \cong \alpha(T) \otimes_{\Sfr} \beta\delta^{-1}(N) \cong \Sfr \otimes_{\Sfr} \beta\delta^{-1}(N),$$  
    This shows that the equivalence $\beta\delta^{-1}: \Disc \Rfr \to \Disc \Sfr$ induces a natural equivalence between the forgetful functors $\Disc {\Rfr} \to \Mod \Zbb$ and $\Disc \Sfr \to \Mod \Zbb$. Then it follows from \cite[Proposition 4.2]{PS19b} that $\Sfr$ and $\Rfr$ are isomorphic as linear topological rings.
\end{proof}
The proof method of \cref{converse} also yields the following neat observation for certain classes of rings.
\begin{prop}\label{noeth2}
    Let $R$ be a left hereditary or a commutative noetherian ring. Suppose there is a Grothendieck category $\Gcal$ and a triangle equivalence $\beta: \D^\bdd(\lMod R) \toeq \D^\bdd(\Gcal)$. Then $\Gcal$ is equivalent to $\Disc \Sfr$, where $T \in \D(\Mod R)$ is a decent tilting complex and $\Sfr = \End_{\D(\Mod R)}(T)$ is its endomorphism ring endowed with the compact topology.
\end{prop}
\begin{proof}
    As in the proof of \cref{converse}, an injective cogenerator $W \in \Gcal$ corresponds under the equivalence $\beta$ to an object $C \in \D^\bdd(\lMod R)$ which is a cotilting complex in $\D(\lMod R)$ and the equivalence $\beta$ restricts to an equivalence $\Hcal_C \cong \Gcal$. By the assumption on $R$, the cotilting complex $C$ is of cofinite type as in \cref{noeth}. The rest follows from (the proof of) \cref{cotilting-cofinite}.
\end{proof}
\begin{rmk}
    In \cref{noeth} and \cref{noeth2}, it is enough to assume that any cotilting complex is of cofinite type. This in turn holds for any ring $R$ such that $\D(\lMod R)$ satisfies certain formulation of the telescope conjecture for t-structures, see \cite[Appendix A]{HN21}. For other examples of classes of rings satisfying this assumption different than left hereditary or commutative noetherian ones, see e.g. \cite[Corollary 3.12, Theorem 8.8]{BH21}.
\end{rmk}
\section{Example: One-dimensional commutative noetherian rings}
In this section, $R$ is a commutative noetherian ring of Krull dimension equal to one. The goal is to work out explicitly how our results apply to tilting and cotilting complexes induced by codimension functions on $\Spec R$, see \cite{HNS} as the basic reference here. In particular, we describe the associated hearts as certain arrow categories.

Let $\mathsf{d}: \Spec R \to \Zbb$ be a codimension function on the Zariski spectrum (here, we can always choose $\mathsf{d}$ to be the height function which assigns the height $\height(\pp)$ to a prime $\pp$). Following \cite[Theorem 4.6]{HNS}, there is a silting complex of the form $T_{\mathsf{d}} = \bigoplus_{\pp \in \Spec R}\RGamma_\pp R_\pp[\mathsf{d}(\pp)]$, where $\RGamma_\pp R_\pp$ is the local cohomology of the local ring $R_\pp$ considered as an object in $\D(\Mod R)$. It follows from \cite[Theorem 6.13]{HNS} and \cref{condition} that $T_{\mathsf{d}}$ is a decent tilting complex. Since the tilting heart and the endomorphism ring of $T_{\mathsf{d}}$ does not depend on the choice of the codimension function $\mathsf{d}$, see \cite[Remark 4.10]{HNS}, we assume $\mathsf{d} = \height$ and denote simply $T = T_{\mathsf{d}}$ from now on. Let $W_i$ be the subset of $\Spec R$ consisting of prime ideals of height $i$ for $i=0,1$. Put $Q = \prod_{\pp \in W_0}R_\pp$ and $\widehat{R} = \prod_{\mm \in W_1}\widehat{R_\mm}$, where $R_\pp$ is the localization at $\pp$ and $\widehat{R_\mm}$ is the $\mm$-adic completion of the local ring $R_\mm$. Note that $Q = R[S^{-1}]$ where $S = R \setminus \bigcup W_0$. In particular, we have the canonical localization flat map $R \to Q$ and the faithfully flat map $R \to \widehat{R}$. We remark that if $R$ is Cohen-Macaulay then $R \to Q$ is precisely the natural map to the total ring of quotients of $R$. In general however, $S$ may contain zero-divisors and so this map can fail to be injective. In any case, we can write the tilting complex as $T = Q \oplus K$, where $K = \Cone(R \to Q)$, see \cite[end of \S 4]{HNS}. From this description, one can also see that $T$ is good. Indeed, there is a triangle of the form $R \to Q \to K \to R[1]$, and so $R \in \thick(T)$.

The endomorphism ring $\Sfr = \End_{\D(\Mod R)}(T)$ can be written explicitly, this is discussed already in \cite[Example 8.4]{PS21} in the case when $Q$ is the total ring of quotients; the non-Cohen-Macaulay situation is covered in \cite[Example 6.8]{HNS}. The endomorphism ring has the following lower triangular matrix presentation:
$$\Sfr = \begin{pmatrix} Q & 0 \\ Q \otimes_R \widehat{R} & \widehat{R}\end{pmatrix}.$$
The compact topology on $\Sfr$ induces the structure of complete and separated topological rings on the corner rings $Q$ and $\widehat{R}$. The ring $Q$ is artinian and as such its topology is the discrete one. The ring $\widehat{R}$ is identified with $\End_{\D(\Mod R)}(K)$, and the primary decomposition $K = \bigoplus_{\mm \in W_1}K_\mm$ implies that the topology on $\widehat{R}$ is the product of topologies on each $\widehat{R_\mm} \cong \End_{\D(\Mod R)}(K_\mm)$. Since $K(\mm)$ represents the local cohomology object $\mathbf{R}\Gamma_\mm R_\mm$, it has a standard expression as a direct limit of the (dual) Koszul complexes supported on $\{\mm\}$. Since Koszul complexes are compact objects, this expression can be used to check that the topology on each $\widehat{R_\mm}$ is none else then the $\mm$-adic one. Note that the ideals of the form $s\widehat{R}$ with $s \in S$ form a base of open ideals for the topology of $\widehat{R}$.

The rest of the topological information is given by a system of open $\widehat{R}$-submodules of the bottom left corner $Q \otimes_R \widehat{R}$; this module identifies with $\Hom_{\D(\Mod R)}(Q,K)$. One can completely describe this topology similarly as in \cite[Example 8.4]{PS21}, but for our purposes it suffices to observe that for any open ideal $J$ of $\widehat{R}$ the left ideal of the form $\begin{pmatrix} 0 & 0 \\ 1 \otimes_R J & \widehat{R}\end{pmatrix}$ is open in $\Sfr$. Indeed, it is enough to check this for $J = s\widehat{R}$ for some $s \in S$, and this left ideal is precisely the annihilator of the map $R \xrightarrow{1 \mapsto s^{-1}} Q \subseteq_{\oplus} Q \oplus K$ from the compact object $R$.

By a standard argument, the category $\lMod \Sfr$ of left $\Sfr$-modules over the lower triangular matrix ring $\Sfr$ can be identified with certain arrow category: The objects are $R$-module maps $V \xrightarrow{\varphi} M$ where $V \in \Mod Q$ and $M \in \Mod \widehat{R}$ and the morphisms are commutative squares
\[
\begin{tikzcd}
V \arrow{r}{\varphi} \arrow{d}{\nu} & M \arrow{d}{\gamma} \\
V' \arrow{r}{\varphi'}& M' 
\end{tikzcd}\]
where $\eta$ is a morphism of $Q$-modules and $\gamma$ a morphism of $\widehat{R}$-modules. The action of an element $q \otimes_R c \in Q \otimes_R \widehat{R} \subseteq \Sfr$ on an object as above takes an element $v \in V$ to $c \varphi(qv) \in M$. The category $\Mod \Sfr$ of right $\Sfr$-modules has an analogous description with arrows $M \xrightarrow{\varphi} V$ going in the opposite direction and the right $\Sfr$-action of an element $q \otimes_R c \in Q \otimes_R \widehat{R}$ is defined using the rule $m(q \otimes_R c) = \varphi(mc)q$ for $m \in M$.

Then the tilting heart $\Hcal_T$ can be described explicitly as follows. By \cref{ab-eq} and \cref{ff-ctra}, we know that $\Hcal_T \cong \Ctra \Sfr$ and $\Ctra \Sfr$ is a full subcategory of $\Mod \Sfr$. Then $\Hcal_T$ identifies with a full subcategory of the above described category of arrows $M \xrightarrow{\varphi} V$ with $M \in \Mod \widehat{R}$ and $V \in \Mod Q$. Furthermore, it is clear from the contraaction of $\Sfr$ that the action of $\widehat{R}$ on $M$ extends to the unique contraaction; note that $\Ctra \widehat{R} \subseteq \Mod \widehat{R}$ is a full subcategory \cite[Corollary 13.13]{Pos16}. On the other hand, for any arrow $\Mfr \to V$ as above with $\Mfr \in \Ctra \widehat{R}$ the right $\Sfr$-action on $\Mfr \oplus V$ extends to a (again, unique) right $\Sfr$-contraaction. The contraaction of the two corner rings of $\Sfr$ is clear: $\widehat{R}$ acts on the $\widehat{R}$-contramodule $\Mfr$ and $Q$ acts as an ordinary ring on $V$. It remains to see how the contraaction is defined given a sequence $(q_\alpha \otimes_R c_\alpha)_{\alpha \in A}$ of elements of $Q \otimes_R \widehat{R}$ which converges to zero in the topology. By the description of the topology above, all but finitely many $q_\alpha$'s can be assumed to be $1$'s. Then for any collection $m_\alpha \in \Mfr, \alpha \in A$, the contramodule action is computed as follows: $\sum_{\alpha \in A}(m_\alpha)(q_\alpha \otimes_R c_\alpha) = \sum_{\alpha \in F}(m_\alpha)(q_\alpha \otimes_R c_\alpha) + \sum_{\alpha \in A \setminus F}(m_\alpha)(1 \otimes_R c_\alpha) = \sum_{\alpha \in F}\varphi(m_\alpha c_\alpha)q_\alpha + \varphi(\sum_{\alpha \in A \setminus F} m_\alpha c_\alpha)$, where $F$ is a finite subset of $A$ such that $q_\alpha = 1$ whenever $\alpha \in A \setminus F$. By the existence of the left open ideals described above, the sequence $(c_\alpha)_{\alpha \in A}$ also converges to zero in the topology of $\widehat{R}$, which ensures that $\sum_{\alpha \in A \setminus F} m_\alpha c_\alpha$ is well-defined using the $\widehat{R}$-contramodule structure of $\Mfr$. One can check directly that all this defines an assignment $\Sfr[[\Mfr \oplus V]] \to \Mfr \oplus V$ satisfying the conditions of contra-associativity and contra-unitality, and thus we obtain the contraaction. Since $\Ctra \Sfr$ is a full subcategory of $\Mod \Sfr$, this shows that $\Hcal_T$ is equivalent to the full subcategory of $\Mod \Sfr$ consisting of those arrows $\Mfr \xrightarrow{\varphi} V$ such that $\Mfr \in \Ctra \Sfr$. We remark that $\Mfr$ decomposes into a product $\prod_{\mm \in W_1}\Mfr(\mm)$ where $\Mfr(\mm) \in \Ctra \widehat{R_\mm}$; for a module-theoretic description of this category see \cite[Corollary 13.13]{Pos16}.

We are left with the task of describing the cotilting heart induced by $C = T^+$. Using \cref{cotilting-eq}, $\Hcal_C$ is equivalent to $\Disc \Sfr$. Here, we claim that $\Hcal_C$ is identified with the full subcategory consisting of arrows of the form $V \xrightarrow{\varphi} M$ where $M$ is in $\Disc{\widehat{R}}$. Indeed, if an arrow $V \xrightarrow{\varphi} M$ belongs $\Disc \Sfr$ then necessarily $M \in \Disc{\widehat{R}}$. On the other hand, consider an element $v \in V$. Then the image $\varphi(v) \in M \in \Disc {\widehat{R}}$ is annihilated by some open ideal $J$ of $\widehat{R}$, and therefore the annihilator of $v$ in $\Sfr$ contains the left open ideal $\begin{pmatrix} 0 & 0 \\ 1 \otimes_R J & \widehat{R}\end{pmatrix}$. It is well-known that $\Disc{\widehat{R}}$ is naturally identified with the full subcategory of $\lMod R$ consisting of modules which are $W_1$-torsion, meaning that they are supported on the set $W_1 \subseteq \Spec R$\footnote{In particular, the connecting $\widehat{R}$-module map $\gamma$ between arrows in $\lMod \Sfr$ representing objects from $\Hcal_C$ is equivalently just an $R$-module map here.}. Yet another description is that these are precisely the modules $M$ such that $M \otimes_R Q = 0$. 

Finally, we remark that the above arrow description of $\Hcal_C$ fits perfectly with the data of the \newterm{Zariski torsion model} constructed in \cite[Remark 8.8, see also Example 8.9, \S 9.3]{torsionmodel} by Balchin, Greenlees, Pol, and Williamson, while the description of $\Hcal_T$ is --- at least seemingly --- different from the \newterm{complete model} of Balchin and Greenlees in \cite{Bsep}, see \cite[\S 10]{Bsep} in particular.
\section*{Acknowledgement}
The author would like to to thank Francesco Genovese, Leonid Positselski, and Jordan Williamson for helpful discussions. A special thanks is due to Leonid Positselski for giving a very useful feedback on an earlier version of the manuscript and to the anonymous referee for carefully reading the manuscript and giving suggestions leading to improvements of the paper.
\bibliographystyle{amsalpha}
\bibliography{bibitems}
\end{document}